\documentclass{article}
\usepackage[T1]{fontenc} 
\usepackage[utf8]{inputenc}
\usepackage{lipsum} 
\usepackage{lmodern}
\usepackage{amssymb}
\usepackage{amsthm}
\usepackage{bm}
\usepackage{mathtools}
\usepackage{bbm,ulem}

\usepackage{braket}
\usepackage{esint}
\usepackage{todonotes}

\usepackage{amsmath}

\usepackage{pgfplots}

\newcommand{\abs}[1]{{\left|#1\right|}}

\newcommand{\norma}[1]{{\left\Vert#1\right\Vert}}

\newcommand{\eps}{\varepsilon}

\newcommand{\parzder}[2]{{\frac{\partial #1}{\partial #2}}}

\usepackage{booktabs}
\usepackage{graphicx}
\usepackage{tikz}
\usetikzlibrary{patterns, math, calc, intersections }
\usepackage{multicol}
\usepackage{caption}
\usepackage{enumerate}
\usepackage[skins,theorems]{tcolorbox}
\tcbset{highlight math style={enhanced,
		colframe=black,colback=white,arc=0pt,boxrule=1pt}}
\def\Xint#1{\mathchoice
    {\XXint\displaystyle\textstyle{#1}}%
    {\XXint\textstyle\scriptstyle{#1}}%
    {\XXint\scriptstyle\scriptscriptstyle{#1}}%
    {\XXint\scriptscriptstyle\scriptscriptstyle{#1}}%
      \!\int}
\def\XXint#1#2#3{{\setbox0=\hbox{$#1{#2#3}{\int}$}
    \vcenter{\hbox{$#2#3$}}\kern-.5\wd0}}
\def\dashint{\Xint-}

\theoremstyle{definition}
\newtheorem{definition}{Definition}[section]

\theoremstyle{plain}
\newtheorem{theorem}{Theorem}[section]

\newtheorem{lemma}[theorem]{Lemma}
\newtheorem{proposition}[theorem]{Proposition}
\newtheorem{corollary}[theorem]{Corollary}
\theoremstyle{definition}
\newtheorem{example}{Example}[section]

\newtheorem{remark}[example]{Remark}

\renewcommand{\div}{\mathrm{div}}
\DeclareMathOperator{\spann}{\text{span}}

\DeclareMathOperator{\R}{\mathbb{R}}
\DeclareMathOperator{\N}{\mathbb{N}}

\newcommand{\1}{\mathbbm{1}}
\newcommand{\Ho}{\mathcal{H}}

\makeatletter
\newcommand{\myfootnote}[2]{\begingroup
	\def\@makefnmark{}%
	\addtocounter{footnote}{-1}%
	\footnote{\textbf{#1} #2}%
	\endgroup}
\makeatother

\numberwithin{equation}{section}

\usepackage{hyperref}
\hypersetup{linktoc=none, bookmarksnumbered, colorlinks=true, linkcolor=magenta, citecolor = cyan}

\usepackage{graphicx} 

\title{Quantitative Kr\"{o}ger inequalities for Neumann eigenvalues of convex domains}
\author{Dorin Bucur\footnote{Universit\'e  Savoie Mont Blanc, Laboratoire de Math\'ematiques CNRS UMR 5127,
  Campus Scientifique, 73376 Le-Bourget-Du-Lac, France
({\tt dorin.bucur@univ-savoie.fr}).}
\and Andrea Gentile  \footnote{Mathematical and Physical Sciences for Advanced Materials and Technologies, Scuola Superiore Meridionale, Largo San Marcellino 10, 80138 Napoli, Italy ({\tt andrea.gentile2@unina.it}).}
\and Antoine Henrot\footnote{Universit\'e de Lorraine, CNRS, Institut Elie Cartan de Lorraine, BP 70239 54506 Vand\oe uvre-l\`es-Nancy Cedex, France ({\tt antoine.henrot@univ-lorraine.fr}).}
}
\date{\today}

\begin{document}

\maketitle

\begin{abstract}
    \noindent Refining the sharp upper bounds $\mu_{k,d}^* $ obtained by Kr\"oger (1999) for the $k$-th Neumann eigenvalue of a convex domain $\Omega \subset \R^d$, we prove the following inequalities: for any $k\in \N$  there exists a constant $C(k,d) >0$ such that 
    $$D_{\Omega}^2 \mu_k(\Omega) \leq \mu_{k,d}^* - C(k,d) a_2(\Omega)^2/D_{\Omega}^2$$
     where $D_{\Omega}$ is the diameter of $\Omega$ and $a_2(\Omega)$ is the second largest semiaxis of the John ellipsoid of $\Omega$. In the planar case, for $k=1$
and with symmetry assumptions, we also give an explicit value of the constant $C(1,2)$.
    
    \medskip

\noindent\textsc{R\'esum\'e.}
Dans cet article, nous donnons une version plus pr\'ecise de la borne sup\'erieure $\mu_{k,d}^* $ obtenue par Kr\"oger (1999) pour la $k$-i\`eme valeur propre Neumann du Laplacien
sur un domaine convexe $\Omega\subset \R^d$. Nous montrons l'in\'egalit\'e suivante:  pour tout $k\in \N$  il existe une constante $C(k,d) >0$ telle que
    $$D_{\Omega}^2 \mu_k(\Omega) \leq \mu_{k,d}^* - C(k,d) a_2(\Omega)^2/D_{\Omega}^2$$
o\`u  $D_{\Omega}$ est le diam\`etre de $\Omega$ et $a_2(\Omega)$ est le deuxi\`eme plus grand semi-axe de l'ellipsoide de John de $\Omega$. Dans le plan, pour $k=1$
et sous des hypoth\`eses de sym\'etrie, nous donnons \'egalement une valeur explicite de la constante $C(1,2)$.

\end{abstract}

\section{Introduction}

Let $\Omega$ be a bounded, open and convex set in $\R^d$, we consider the  Neumann-Laplacian eigenvalue problem
\begin{equation}
    \label{eigenvalue_problem_PDE}
    \begin{cases}
        - \Delta u = \mu u &\text{ in } \Omega \\[0.5ex]
        \displaystyle{\frac{\partial u}{\partial \nu} = 0} &\text{ on }\partial \Omega
    \end{cases},
\end{equation}
where $\nu$ is the outer normal to $\Omega$.
The embedding of $H^1(\Omega)$ into $L^2(\Omega)$ is compact, hence the   spectrum is discrete and consists only on eigenvalues (ordered with multiplicity)
\[
    0 = \mu_0(\Omega) < \mu_1(\Omega) \leq \mu_2(\Omega) \leq \ldots \leq \to +\infty,
\]
Each eigenvalue can be characterized by
\begin{equation}
    \label{eq_mu_k}
    \mu_k(\Omega) = \min_{E_k} \max_{0 \neq u \in E_k} \frac{\displaystyle \int_{\Omega} \abs{\nabla u}^2 \, dx }{\displaystyle \int_{\Omega} u^2 \, dx },
\end{equation}
where the minimum is taken over all the $k$-dimensional subspaces of $H^1(\Omega)$, $L^2$-orthogonal to the constant functions.

%
%

The spectrum of the Laplacian has been extensively studied in the literature in relation to the geometry of the domain. A natural question arising in connection 
with Pólya’s conjecture is the following maximization problem:\begin{equation}
    \label{max_problem_Szego_Weinberger}
    \sup \{ \mu_k(\Omega) \abs{\Omega}^{\frac{2}{d}} \colon \Omega \subset \R^d, \, \Omega \text{ open, bounded and Lipschitz} \},
\end{equation}
where $\lvert \cdot \rvert$ denotes the $d$-dimensional Lebesgue measure. The functional above is scale invariant since
$ \mu_k(t\Omega) = t^{-2}\mu_k(\Omega)$, for all $t >0$.

This problem was solved by Szeg\"{o} \cite{Szego} for planar, simply connected domains, while Weinberger \cite{Weinberger} generalized the result in $\R^d$ removing any topological restrictions :
\begin{equation}
    \mu_1(\Omega) \abs{\Omega}^{\frac{2}{d}} \leq \mu_1(B) \abs{B}^{\frac{2}{d}},
\end{equation}
where $B$ is a ball in $\R^d$.

In \cite{Girouard_Nadirashvili_Polterovich}  the authors find the sharp upper bound for $\mu_2$  for simply connected planar domains, and in  \cite{Bucur_Henrot_second_eigenvalue} in any dimension without topological restrictions, obtaining
\begin{equation}
    \mu_2(\Omega) \abs{\Omega}^{\frac{2}{d}} \leq 2^\frac 2d\mu_1(B) \abs{B}^{\frac{2}{d}}.
\end{equation}
For higher eigenvalues, even the existence of maximizer of the quantity $\mu_k(\Omega) \abs{\Omega}^{\frac{2}{d}}$ is an open problem. See \cite{Bucur_Martinet_Oudet} for an analysis of $\mu_k(\Omega)$ for $k\geq 3$
and \cite{Bo-He-Mi24} for the optimization of $\mu_k(\Omega)$ among planar convex domains with perimeter or diameter constraints.

In the context of optimizing Neumann eigenvalues, another natural question concerns the relationship between eigenvalues and the diameter. Kröger \cite{Kro99} studied the maximization of the $k$-th Neumann eigenvalue among convex sets $\Omega \subset \mathbb{R}^d$, $d \geq 2$, with prescribed diameter $D_\Omega$, and obtained sharp upper bounds.$\mu_{k,d}^*$  such that
\begin{equation}
    \label{eq:Henrot_Michetti_e_Kroger_intro}
    D_{\Omega}^2\mu_k(\Omega) <\mu_{k,d}^* \qquad \forall \, \Omega \subseteq \R^d, \, \Omega \text{ open, bounded and convex}.
\end{equation}
A variational analysis of these bounds together with a fine description of a maximizing sequence were given in \cite{Henrot_Michetti}.

We point out that equality in \eqref{eq:Henrot_Michetti_e_Kroger_intro}   is not attained for $d\ge 2$. Instead, it is asymptotically approached by  a   sequence of collapsing domains  
with  a specific geometry, which converges to a segment. From this perspective, there is a strong connection between this problem and the analysis of eigenvalues of a certain class of Strum-Liouville problems. 
We also mention that the case $k=1$ was previously considered in \cite{Cheng75}, \cite{Ban-Bur99}
and was also studied in a nonlinear framework in \cite{Brasco_Nitsch_Trombetti}.

Given a sharp inequality such as \eqref{eq:Henrot_Michetti_e_Kroger_intro}, a natural question is what can be said about domains whose eigenvalue is close to the upper bound. More precisely, one may seek a quantitative version of the above inequality, involving a term that measures the "distance" between the domain and the limiting segment.

 The purpose of this paper is to give such sharp quantitative inequalities, naturally involving  the flatness  of the domain.
 Several examples of quantitative estimates  can be found  in  the literature. See for instance \cite{FMP_quantitative_isoperimetric,BdPV} and the survey paper \cite[Chapter 7]{Shape_Optimization_and_spectral_theory}. Typically, in all these problems  the target set is a ball and the "distance" to the  ball is expressed by the Fraenkel asymmetry.
Let us emphasize that the main difference between the aforementioned estimates and ours is that, in our case, the inequality is attained only asymptotically along a sequence of thinning domains, so that no optimal geometry exists. Results of this type are relatively rare in the literature (see, for instance, \cite{AMPS, ama-buc-fra1, ama-buc-fra2}), and the existing ones concern only the first and, in some cases, the second eigenvalue. In \cite{ama-buc-fra1}, it is proved that there exists a dimensional constant $C(d)>0$ such that, for any open, bounded, convex set $\Omega \subset \mathbb{R}^d$,
\begin{equation}\label{d30}
        \mu_1(\Omega) \geq \frac{\pi^2}{D_{\Omega}^2} + C(d) \frac{a_2(\Omega)^2}{D_{\Omega}^4}.
\end{equation}
where $D_{\Omega}$, $a_2(\Omega)$ are respectively the diameter and the second largest semiaxis of the John ellipsoid 
    of $\Omega$ (see Section \ref{section_2} for the definition of John ellipsoid).

In this paper we   establish the following quantitative estimates of \eqref{eq:Henrot_Michetti_e_Kroger_intro} : 
\begin{theorem}
    \label{main_theorem_k=2}
    Let $k,d \in \N$, $d\ge 2$. There exist a constant $C(k,d) > 0$ depending only on $k$ and $d$  such that for any $\Omega \subset \R^d$ open, bounded and convex set, it holds
    \begin{equation}
        \label{eq_quantitative_mu_2}
        \mu_k(\Omega) \leq \frac{\mu_{k,d}^*}{D_{\Omega}^2} - C(k,d)  \frac{a_2(\Omega)^2}{D_{\Omega}^4},
    \end{equation}
    where $\mu_{k,d}^*$ is the sharp  upper bound for $\mu_k(\Omega) D_{\Omega}^2$ given by \cite{ Kro99,Henrot_Michetti}.
\end{theorem}
Moreover we establish that  the exponent $2$ on the flatness term can not be improved, in general. 
\begin{theorem}
    \label{teo:optimality_exponent_2}
    Exponent $2$ in \eqref{eq_quantitative_mu_2} over the second John dimension is sharp for $d=2$ and $k=1$.
\end{theorem}

Since the proof of \eqref{eq_quantitative_mu_2} is done by contradiction, it does not provide an explicit value of the constant $C(k,d)$.
Let us remark that it is the case in most quantitative isoperimetric inequalities, see e.g. \cite{FMP_quantitative_isoperimetric,BdPV} and the survey paper \cite[Chapter 7]{Shape_Optimization_and_spectral_theory}. Therefore, we propose in  Section \ref{section_4} to give an explicit value
of the constant $C(1,2)$ (at least for domains with two axis of symmetry). We prove
\begin{theorem}
Let $\Omega$ be a convex planar domain and assume that $\Omega$ is symmetric with respect to the mediatrix of its diameter and with respect to an orthogonal line , then
\begin{equation}
D_{\Omega}^2 \mu_1(\Omega) \leq 4j_{0,1}^2  - 0.108 \frac{w_\Omega^2}{D_{\Omega}^2}
\end{equation}
where $w_\Omega$ is the width of the convex domain in the direction orthogonal to the diameter.
\end{theorem}
The value $0.108$ is not sharp but, in order to make a comparison, it is worth to notice that the estimate of the constant $C(2)$ in \eqref{d30} is of order $10^{-16}$ (see \cite{ama-buc-fra1}) while its numerical estimate ranges between $5$ and $6$. The sharp value of a constant similar to $C(2)$ is only known in the purely geometric case corresponding to the $W^{1,1}$ space, and equals $\frac 43$ (see \cite{BuFr26}). 

The paper is organized as follows: in Section \ref{section_2} we recall some notions and properties that will be used for the proof of Theorem \ref{main_theorem_k=2} while Section \ref{sec:proof_of_theorems} is devoted to the proof of the  Theorems \ref{main_theorem_k=2} and \ref{teo:optimality_exponent_2}. Lastly in Section \ref{section_4} we compute an explicit value of the constant $C(1,2)$ by solving an interesting problem in the calculus of variations.

\section{Notations and preliminaries}
\label{section_2}

In this section we recall some useful results.

\medskip
\noindent{\bf Convex geometry.}
We start by recalling the Brunn-Minkowski inequality (see \cite{Gardner_Brunn_Minkowski}, \cite{Schneider}). 
\begin{theorem}[Brunn-Minkowski inequality]
    \label{theo:Brunn_Minkowski}
    Let $\Omega_1, \Omega_2 \subset \R^d$ two compact convex sets with non empty interior, then for each $0 < \lambda <1$ it holds
    \begin{equation}
        \label{eq:Brunn_Minkowski_inequality}
        \mathcal{L}^d \bigl( \lambda \Omega_1 + (1-\lambda)\Omega_2 \bigr)^{\frac{1}{d}} \geq  \lambda \mathcal{L}^d (\Omega_1)^{\frac{1}{d}} + (1-\lambda) \mathcal{L}^{d}(\Omega_2)^{\frac{1}{d}},
    \end{equation}
    where $\cdot + \cdot$ denotes the Minkowski sum among two sets and $\mathcal{L}^d$ the $d$-dimensional Lebesgue measure.
Moreover equality in \eqref{eq:Brunn_Minkowski_inequality} holds if and only if $\Omega_1$ and $\Omega_2$ are homothetic. 
\end{theorem}

Let $\Omega$ be a open, bounded and convex set in $\R^d$, and $ j<d$. We define for each $x \in \R^j$ the section
\[
    S_x = \{ x' \in \R^{d-j} \, \colon \, (x,x') \in \Omega \}.
\]
We denote with $\omega \subseteq \R^j$ the projection of $\Omega$ onto the first $j$ coordinates and define
\[
 p:\omega \to \R, \;\;    p (x) = \Ho^{d-j}\bigl( S_x \cap \Omega \bigr).
\]
In case $j=1$, the function $p^\frac{1}{d-1}$ is usually referred as the  {profile function} of $\Omega$.

As consequence of the Brunn-Minkowski inequality, we have the following
\begin{corollary}
    \label{corollary_Brunn_Minkowski}
    Let $\Omega \subset \R^d$ be a convex set and $p:\omega\to \R$ defined as above. Then $p$ is $\frac{1}{d-j}$-concave, i.e. $p^\frac{1}{d-j}$ is concave.
\end{corollary}

We will consider the Hausdorff distance on the class of convex sets.
\begin{definition}
    Let $\Omega_1, \, \Omega_2 \subset \R^d$ two convex sets. The  {Hausdorff distance} between $\Omega_1$ and $\Omega_2$ is defined as
    \begin{equation}
        \label{Hausdorff_distance}
        d_{\mathcal{H}}(\Omega_1,\Omega_2) = \inf \{ \eps > 0 \, : \, \Omega_1 \subset \Omega_2 +\eps B_1, \, \Omega_2 \subset \Omega_1 + \eps B_1 \},
    \end{equation}
    where $B_1$ is the ball of radius $1$ in $\R^d$.
\end{definition}

We recall the following compactness property of the Hausdorff distance (known as Blaschke selection theorem), see e.g.\cite[Theorem 2.2.25]{Henrot_Pierre}.
\begin{theorem}
   Let $\{ \Omega_n \}$ be a sequence of convex sets contained in a fixed compact set $K \subset \R^d$. Then, there exist a convex set $\Omega \subset K$ and a subsequence, still denoted by $\{ \Omega_n \}$, such that $\Omega_n$ converges to $\Omega$ in the Hausdorff sense.
\end{theorem}

\medskip
\noindent{\bf The Sturm-Liouville problem}.
 Let $m \in \N$, $\omega \subseteq \R^j$ be an open, bounded and convex set, $p \colon \omega \to \R$,  be a non negative $\frac{1}{m}$-concave function. 
 Then, the  embedding of $W^{1,2}(\omega,p)$ into $L^2(\omega,p)$ is compact (see \cite[Lemma 7]{ama-buc-fra1}). Consequently,  the following problem
\begin{equation}
    \label{eq_Sturm_Liouville}
    \begin{cases}
        \displaystyle{ - \div \bigl( p(x) \nabla u(x) \bigr) = \mu_k(\omega,p) p(x) u(x) } & \text{ in } \omega \\[1.4ex]
        \displaystyle{ p(x) \parzder{u}{\nu}= 0} & \text{ on }\partial \omega.
    \end{cases}
\end{equation}
is well posed in a weak (variational) sense and has a sequence of eigenvalues, ordered with multiplicity
$$0=\mu_0(\omega,p)<\mu_1(\omega,p)\le\dots  \mu_k(\omega,p)\le \dots +\infty$$
defined by
  \begin{equation}
        \label{eq_mu_k_h}
        \mu_k(\omega,p) = \inf_{E_k} \sup_{0 \neq u \in E_k} \frac{\displaystyle \int_{\omega} \abs{\nabla \varphi (x)}^2 p(x) \, dx }{\displaystyle \int_{\omega} \varphi(x)^2 p(x) \, dx}.
    \end{equation}
The infimum above is taken among all the $k$-dimensional subspace of $H^1(\omega)$ that are $L^2$-orthogonal to constant functions with respect to $p$ on $\omega$.

If $j = 1$, problem \eqref{eq_Sturm_Liouville} is   a classical Sturm-Liouville problem. In this case we simply denote 
$$\mu_k(\omega, p)= \mu_k(p).$$

\medskip
\noindent{\bf Collapsing domains}.
Assume that $\{ \Omega_n \}$ is a sequence of bounded, open and convex sets in $\R^d$ with diameter $D_{\Omega_n} = 1$. We  denote $E_n$ the John ellipsoid (see \cite{Guzman_differentiation}) of $\Omega_n$, which satisfies, up to a translation, $E_n \subseteq \Omega_n \subseteq dE_n$. Up to a rigid movement we  can assume that for each $n \in \N$
\begin{enumerate}
    
    \item $E_n$ is centered at the origin;

    \item $E_n$ is aligned with the axes;

    \item the semi-axis satisfy
    \[
        1> a_1(\Omega_n) \geq a_2(\Omega_n) \geq \ldots \geq a_d(\Omega_n) > 0.
    \]
\end{enumerate}

Up to a subsequence we can assume that
\[
    a_i(\Omega_n) \to a_i \qquad \forall i \in \{ 1,\ldots,d \}
\]
such that
\[
    a_1 \geq \ldots \geq a_{j} > 0 = a_{j+1} = \ldots = a_d.
\]
The sequence $\{ \Omega_n \}$  is collapsing if  $ j<d$.

Thanks to the compactness property of the Hausdorff distance, we can also assume
\begin{equation}
    \label{eq:convergence_Omega_n_to_omega_times_0}
    \Omega_n \overset{\mathcal{H}}{\longrightarrow} \omega \times \{ 0 \}^{d-j},
\end{equation}
where $\omega \subseteq \mathbb{R}^j$ is an open, convex set of diameter $1$.

For each $n$ we introduce
\begin{equation}
    \label{eq:def_T_n}
    T_n \colon \R^d \to \R^d,
    \qquad
    T_n(x_1,\ldots,x_d) = \biggl( x_1,\ldots, x_j, \frac{x_{j+1}}{a_{j+1}(\Omega_n)},\ldots, \frac{x_d}{a_d(\Omega_n)} \biggr).
\end{equation}
Essentially, when applied to $\Omega_n$, this function rescales the last $d-j$ components in such a way that $\tilde \Omega_n :=T_n(\Omega_n)$ has a John  ellipsoid with semiaxes
\[
    a_1(\Omega_n),\ldots , a_j(\Omega_n), 1,\ldots,1.
\]
So in particular, $\tilde \Omega_n $ is a non degenerating sequence of convex open sets which can be assumed to converge to a set $\tilde \Omega$. Moreover, the projection of $\tilde \Omega$ on the first $j$ variables is $\omega$. 
We introduce $
        p \colon \omega \to \R$, $p(x) = \Ho^{d-j}(S_x \cap \tilde{\Omega})$, $ \forall x \in \omega$.
    Corollary \ref{corollary_Brunn_Minkowski} ensures that $p$ is a $\frac{1}{d-j}$- concave  function.

We recall the following result from \cite[Proposition 8]{ama-buc-fra1}.
\begin{theorem}
    \label{theo:ABF_24}
Under the previous assumptions    \[
        \mu_k(\Omega_n) \to \mu_k(\omega,p),
    \]
    where $\mu_k(\omega,p)$ is the $k$-th eigenvalue of   problem \eqref{eq_Sturm_Liouville}.
       Moreover, if $(u_k^n)$  is a sequence of $L^2$-normalized eigenfunctions associated to $\mu_k(\Omega_n)$, then defining
 \[
        \tilde u_k^n \colon \tilde{\Omega}_n \to \R,
        \qquad
        \tilde u_k^n(x_1,\ldots,x_d) = k_n u_k^n\bigl( x_1,\ldots, x_j, a_{j+1}(\Omega_n) x_{j+1},\ldots, a_d (\Omega_n) x_d \bigr),
    \]
    where $k_n = \sqrt{\prod_{h=j+1}^d a_h(\Omega_n) }$,
we get, up to extracting a subsequence, that 
$$1_{\tilde \Omega_n} \tilde u_k^n \to 1_{\tilde \Omega}   \tilde u_k, \mbox{ strongly in } L^2(\R^d),$$
$$ 1_{\tilde \Omega_n} \nabla \tilde u_k^n \to 1_{\tilde \Omega}\nabla \tilde u_k, \mbox{ strongly in } L^2(\R^d),$$
where $ \tilde u_k(x_1, \dots, x_d):= u_k(x_1, \dots, x_j)$
is a k-th eigenfunction of  problem \eqref{eq_Sturm_Liouville} set on $(\omega, p)$.
\end{theorem}
\begin{proof}
    We briefly recall the main idea of the proof, since it is useful for our future  analysis (we refer for the full proof  to \cite[Proposition 8]{ama-buc-fra1}). 
    It holds
    \begin{align*}
        \norma{\tilde u _k^n}_{L^2(\tilde{\Omega}_n)}^2 & = \norma{u_k^n}_{L^2(\Omega_n)}=1,
    \end{align*}
    and 
 $$
     \forall i=1, \dots, j\; \int_{\tilde{\Omega}_n} \biggl( \parzder{\tilde u_k^n}{x_i} \biggr)^{2}   \, dx =  \int_{\Omega_n} \biggl( \parzder{u_k^n}{x_i} \biggr)^{2}    \, dx $$
     $$  \forall i=j+1, \dots, d\; \int_{\tilde{\Omega}_n} \biggl( \parzder{\tilde u_k^n}{x_i}  \biggr)^{2}    \, dx = \bigl( a_i(\Omega_n) \bigr)^2 \int_{\Omega_n} \biggl( \parzder{u_k^n}{x_i}  \biggr)^{2}   \, dx.
  $$
    Since $\tilde \Omega_n$ are uniformly convex and $\tilde u_{n}^k \in H^1 \bigl( \tilde{\Omega}_n \bigr)$ are uniformly bounded in $H^1$ we can assume that there exists $\tilde u_k \in H^1(\tilde{\Omega})$ such that (see \cite[Section 3.7]{Henrot_Pierre})
    \begin{align*}
        \tilde u _k^n \1_{\tilde{\Omega}_n} & \to \tilde u_k \1_{\tilde{\Omega}} & & \text{ in }L^2 \\
        \bigl( \nabla \tilde u_k^n \bigr) \1_{\tilde{\Omega}_n} & \to \nabla \tilde u_k \1_{\tilde{\Omega}} && \text{  weakly  in }L^2.
    \end{align*}
    We  conclude with
     \[
        \tilde u_{k}(x,x') = u_k(x),
     \]
   meaning  that is $\tilde u_k$ is constant in variables $x'$. Following \cite[Proposition 8]{ama-buc-fra1}, it is shown that $ \tilde u_k$ is the $k$-th eigenfunction of problem \eqref{eq_Sturm_Liouville} and that
    \[
        \lim_{n \to +\infty} \mu_k(\Omega_n) = \mu_k(\omega,p) \qedhere.
    \]
 The convergence of the gradients is strong in $L^2$ as a consequence of the convergence of their norms. 
\end{proof}

 \begin{remark}\label{d31}
 In the particular case in which $j=1$ and $\omega =(0,1)$, the sets $\Omega_n$ which have diameter equal to $1$ collapse to $(0,1)\times \{0\}^{d-1}$. After a possible scaling factor (converging to 1) it can be assumed that $(0,1)\times \{0\}^{d-1}$ is the projection of $\tilde \Omega_n$ onto the $x_1$-axis. One can define  $
        p_n: \omega \to \R$, $p_n(x) = \Ho^{d-1}(S_x \cap \tilde \Omega_n)$, $ \forall x \in \omega$. Then $\tilde p_n^\frac{1}{d-1}$ are concave and converge locally uniformly on $(0,1)$ to $p^\frac{1}{d-1}$. 
        
 Then, from the variational formulation, we readily get
 $$\mu_k(p_n) \ge \mu_k(\Omega_n).$$

The statement of Theorem \ref{theo:ABF_24} continues to hold, namely 
$$\mu_k(p_n) \to \mu_k(p),$$ 
so that $\mu_k(p_n) -\mu_k(\Omega_n ) \to0$. Moreover, 
the convergence of eigenfunctions associated to $\mu_k(p_n) $ holds as well. We refer to the proof of
 \cite[Proposition 8]{ama-buc-fra1} and to \cite[Lemma 2.3]{Henrot_Michetti}.
 \end{remark}

Let us introduce the following notation: if $\{ \Omega_n \}$ is a sequence of open, bounded and convex sets converging in the sense of Hausdorff to $\omega \times \{ 0 \}^{d-j}$ where $\omega$ is a bounded, open and convex set of $\R^j$ and $p$ is the function given by Theorem \ref{theo:ABF_24}, we will write
\[
    \Omega_n \to (\omega,p).
\]

\medskip
\noindent{\bf Geometry of convex sets with affine profile functions}.
Let us assume that $\Omega$ is an open, bounded and convex set of diameter $(0,1)\times \{0\}^{d-1}$ and  profile function of the form
\begin{equation}
    \label{eq:expression_p^d-1}
    p(x)^{\frac{1}{d-1}}
    =
    \begin{cases}
        c_1x & \text{ if } x \in [0,\alpha] \\
        ax+b & \text{ if } x \in [\alpha,\beta] \\
        c_2(1-x) & \text{ if } x \in [\beta,1],
    \end{cases}
\end{equation}
where $\alpha, \beta,a,b,c_1,c_2$ are such that $p(x)$ is continuous and $\frac{1}{d-1}$-concave (Figure \ref{fig:graph_of_p}).

\begin{figure}[h!]
    \centering
    \begin{tikzpicture}
        \draw[-latex] (-0.5,0) -- (5,0);
        \draw[-latex] (0,-0.5) -- (0,2.5);

        \filldraw[black] (0,0) circle(1pt) node[anchor = 45] {$0$};
        \filldraw[black] (1.5,0) circle(1pt) node[anchor = north] {$\alpha$};
        \filldraw[black] (3,0) circle(1pt) node[anchor = north] {$\beta$};
        \filldraw[black] (4.5,0) circle(1pt) node[anchor = north] {$1$};

        \draw[black] (0,0) -- (1.5,1.5) -- (3,2) -- (4.5,0);
        \draw[black, dashed] (1.5,1.5) -- (1.5,0);
        \draw[black, dashed] (3,2) -- (3,0);
    \end{tikzpicture}
    \caption{}
    \label{fig:graph_of_p}
\end{figure}
Although the profile function merely measures the volume of the convex set's sections, the linearity of $p^{\frac{1}{d-1}}$  imposes a strict geometric rigidity on $\Omega$. Specifically, this linearity yields equality in the Brunn-Minkowski inequality \eqref{eq:Brunn_Minkowski_inequality} for all sections within the intervals  $[0,\alpha]$, $[\alpha,\beta]$ and $[\beta,1]$,   respectively. The characterization of the equality cases in the Brunn-Minkowski inequality (see, e.g., \cite{Gardner_Brunn_Minkowski}) implies that the sections of $\Omega$  within the intervals $[0,\alpha]$, $[\alpha,\beta]$ and $[\beta,1]$  are homothetic. Combined with the convexity of $\Omega $, this ensures that within each slab  $[0,\alpha] \times \R^{d-1}, [\alpha,\beta] \times \R^{d-1}$ and $[\beta,1] \times \R^{d-1}$, the set $\Omega$ a truncated cone (Figure \ref{fig:truncation_of_a_cone}).

\begin{figure}[h!]
    \centering
    \begin{tikzpicture}
        \draw[-latex] (0,0) -- (8.5,0);
        \draw[-latex] (0,0) -- (0,2.5);
        \draw[-latex] (0,0) -- (-120:3);

        \node[black] at (8.3,-0.3) {$x_1$};
        \node[black] at (-114:2.7) {$x_2$};

        \draw[cyan] (2.5,-0.5) ellipse(0.4 and 1.2);
        \draw[cyan] (5,-0.9) ellipse(0.4 and 1.5);

        \draw[cyan] (0,0) -- (2.5,0.7) -- (5,0.6) -- (7.5,0);
        \draw[cyan] (0,0) -- (2.5,-1.7) -- (5,-2.4) -- (7.5,0);

        \filldraw[black] (0,0) circle(1pt) node[anchor = east] {$0$};
        \filldraw[black] (2.5,0) circle(1pt) node[anchor = north] {$\alpha$};
        \filldraw[black] (5,0) circle(1pt) node[anchor = north] {$\beta$};
        \filldraw[black] (7.5,0) circle(1pt) node[anchor = north] {$1$};

    \end{tikzpicture}
    \caption{}
    \label{fig:truncation_of_a_cone}
\end{figure}

For instance in $[\alpha,\beta]$, because of the homothety of $S_{\beta} \cap \Omega$ and $S_{\alpha} \cap \Omega$ there exist $k \in \mathbb{R}$ and $m \in \mathbb{R}^{d-1}$ such that
\[
    S_{\beta} \cap \Omega = k S_{\alpha} \cap \Omega + m.
\]
Following the expression of $p$, we deduce
\[
    k = \biggl( \frac{p(\beta)}{p(\alpha)} \biggr)^{\frac{1}{d-1}}.
\]
Moreover, since $\Omega$  is convex, the set $( [\alpha, \beta] \times \R^{d-1} ) \cap \Omega$ contains the cone passing through $S_{\alpha} \cap \Omega$ and $S_{\beta} \cap \Omega$ and because of the linearity of the profile function, the sets must coincide.

Let us underline that the previous analysis guarantees that for each $x_1 \in [\alpha,\beta]$ it holds
\begin{equation}
    \label{eq:per_change_variable_sections}
    S_{x_1} \cap \Omega = \biggl[ \frac{p(x_1)}{p(\alpha)} \biggr]^{\frac{1}{d-1}} \bigl( S_{\alpha} \cap \Omega \bigr) + \frac{x_1-\alpha}{\beta-\alpha} m, \qquad \forall x_1 \in [\alpha,\beta] \subset [0,1].
\end{equation}
Now let us define the  function
\[
    c(x_1) \coloneqq \dashint_{S_{x_1} \cap \Omega} x_2 \, d\Ho^{d-1},
\]
where $\dashint$ denotes the averaged integral. The function $c$ is  piecewise affine. Indeed,  for $x_1 \in [\alpha, \beta]$,  a change of variable using \eqref{eq:per_change_variable_sections} yields
\begin{align*}
    \int_{S_{x_1} \cap \Omega} x_2 \, d\Ho^{d-1}
    & =
    \frac{p(x_1)}{p(\alpha)} \int_{S_{\alpha} \cap \Omega} \biggl[  \biggl( \frac{p(x_1)}{p(\alpha)}\biggr)^{\frac{1}{d-1}} y_2  + \biggl( \frac{x_1-\alpha}{\beta-\alpha} m_2 \biggr) \biggr] \, d\Ho^{d-1} \\
    & =
    p(x_1) \Biggl( \frac{1}{p(\alpha)} \int_{S_{\alpha} \cap \Omega} \frac{y_2}{ p(\alpha)^{\frac{1}{d-1}}} \, d \Ho^{d-1} \Biggr) \, p(x_1)^{\frac{1}{d-1}}
    +
    m_2\biggl( \frac{x_1-\alpha}{\beta-\alpha} \biggr) p(x_1),
\end{align*}
therefore
\begin{equation}
    \label{eq:espressione_c}
    c(x_1) = c(\alpha) \biggl( \frac{p(x_1)}{p(\alpha)} \biggr)^{\frac{1}{d-1}} + m_2\biggl( \frac{x_1-\alpha}{\beta-\alpha} \biggr),
\end{equation}
is an affine function since $p(x)^{\frac{1}{d-1}}$ is an affine function.  

We now show that there exists  $K$ such that
\begin{equation}
    \label{eq:momento_di_ordine_2}
    \int_{S_{x_1}\cap \Omega} \bigl( x_2 - c(x_1)\bigr)^2 \, d\Ho^{d-1} = Kp(x_1)^{1+\frac{2}{d-1}} \qquad \forall x_1 \in [\alpha,\beta].
\end{equation}

First of all we notice that \eqref{eq:momento_di_ordine_2} can be written as
\[
    \dashint_{S_{x_1}\cap \Omega} \bigl( x_2 - c(x_1)\bigr)^2 \, d\Ho^{d-1} = Kp(x_1)^{\frac{2}{d-1}},
\]
and we can write
\begin{align*}
    \dashint_{S_{x_1}\cap \Omega} \bigl( x_2 - c(x_1)\bigr)^2 \, d\Ho^{d-1}
    & =
    \dashint_{S_{x_1}\cap \Omega} x_2^2 \, d\Ho^{d-1} -2c(x_1) \dashint_{S_{x_1} \cap \Omega} x_2 \, d\Ho^{d-1} + c(x_1)^2 \\
    & = \dashint_{S_{x_1}\cap \Omega} x_2^2 \, d\Ho^{d-1} -c(x_1)^2,
\end{align*}
where we used the definition of $c(x_1)$.

Let us now compute $\dashint_{S_{x_1} \cap \Omega} x_2^2 \, d\Ho^{d-1}$ with the same change of variable as before:
\begin{align*}
    \dashint_{S_{x_1}\cap \Omega} x_2^2 \, d\Ho^{d-1}
    & =
    \frac{1}{p(\alpha)} \int_{S_{\alpha} \cap \Omega} \biggl[ \biggl( \frac{p(x)}{p(\alpha)} \biggr)^{\frac{1}{d-1}} y_2 + \biggl( \frac{x_1-\alpha}{\beta-\alpha} m_2 \biggr)  \biggr]^2  \, d \Ho^{d-1} \\
    & =
    \biggl( \frac{p(x_1)}{p(\alpha)} \biggr)^{\frac{2}{d-1}} \dashint_{S_{x_1} \cap \Omega} y_2^2 \,  d\Ho^{d-1}
    +
    m_2^2\biggl( \frac{x_1-\alpha}{\beta-\alpha} \biggr)^2
    +
    2 m_2 \biggl( \frac{x_1-\alpha}{\beta-\alpha} \biggr) \biggl( \frac{p(x_1)}{p(\alpha)} \biggr)^{\frac{1}{d-1}} c(\alpha).
\end{align*}
Hence, using \eqref{eq:espressione_c}, we get
\begin{align*}
    \dashint_{S_{x_1}\cap \Omega} \bigl( x_2 - c(x_1)\bigr)^2 \, d\Ho^{d-1}
    =
    \biggl( \frac{p(x_1)}{p(\alpha)} \biggr)^{\frac{2}{d-1}} \biggl[ \dashint_{S_{\alpha} \cap \Omega} y_2^2 \, d\Ho^{d-1} - c(\alpha)^2 \biggr] = Kp(x_1)^{\frac{2}{d-1}},
\end{align*}
where
\[
    K \coloneqq \frac{1}{p(\alpha)^{\frac{2}{d-1}}} \biggl[ \dashint_{S_{\alpha} \cap \Omega} y_2^2 \, d\Ho^{d-1} - c(\alpha)^2 \biggr],
\]
so \eqref{eq:momento_di_ordine_2} holds true.

     Of course, the same computation can be done on  $[0,\alpha]$ and $[\beta,1]$ and, as we will detail in the proof of Theorem \ref{main_theorem_k=2}, the constant $K$ is the same on all intervals. It is  constant on each interval and continuous at the  endpoints of the intervals.

\medskip
\noindent{\bf 
Extension of Kr\"{o}ger inequalities to weighted eigenvalues}.
We recall the following result from \cite{Kro99} and \cite{Henrot_Michetti}.
\begin{theorem}
    \label{theo:Henrot_Michetti}
  Let $k,d \in \N$, $d\ge 2$, $k\ge 1$. Then for every $\Omega \subseteq \R^d$, $\Omega$  open, bounded and convex
 \begin{equation}
    \label{eq:Henrot_Michetti_e_Kroger_intro.bis}
    D_{\Omega}^2\mu_k(\Omega) <\mu_{k,d}^* \qquad \end{equation}
    where the  sharp value $\mu_{k,d}^* $ equals
    \begin{equation}
        \label{eq:sup_maximization_problem_p_d}
         \mu_{k,d}^*=\sup \biggl\{ \mu_k(p) \, \colon \, p \in L^{\infty}(0,1), \, p \geq 0, \, p \text{ is } \frac{1}{d-1} \text{ concave, } p \not \equiv 0  \biggr\}.
    \end{equation}
Moreover
\begin{itemize}
    \item if $d = 2$ then 
    \[
    \mu_{k,d}^*= \left(2 j_{0,1} + (k - 1)\pi \right)^2
    \]
    
    \item if $d = 3$ then 
    \[
    \mu_{k,d}^*= \left((k + 1)\pi \right)^2
    \]
    
    \item if $d \geq 4$ then:
    \begin{itemize}
        \item if $k$ is odd then 
        \[
        \mu_{k,d}^*=4\, j_{\frac{d-2}{2},\,\frac{k+1}{2}}^2
        \]
        
        \item if $k$ is even then 
        \[
        \mu_{k,d}^*=\left(
        j_{\frac{d-2}{2},\,\frac{k}{2}} + 
        j_{\frac{d-2}{2},\,\frac{k+2}{2}}
        \right)^2
        \]
    \end{itemize}
\end{itemize}
Denoting by $p_{k,d}^*$ a function for which  the maximum is achieved in \eqref{eq:sup_maximization_problem_p_d},
    \begin{itemize}
        \item If $k = 1$, then the graph of $(p_{1,d}^*)^{\frac{1}{d-1}}$ is an isosceles triangle having vertex in $x=\frac{1}{2}$.

        \item If $k \geq 2$, then $(p_{k,d}^*)^{\frac{1}{d-1}}$ is a piecewise affine function such that
        \begin{itemize}
        \item if $d = 3$, any function such that $p_{k,3}^*(0)=p_{k,3}^*(1) = 0$ and the graph of $(p_{k,3}^*)^{\frac{1}{2}}$ is made by at most $k+1$ segments;

        \item if $d \neq 3$, the graph of $(p_{k,d}^*)^{\frac{1}{d-1}}$ is an isosceles trapezoid.
    \end{itemize}

    \end{itemize}

\end{theorem}

%
%
%
%

We now extend Theorem \ref{theo:Henrot_Michetti} to weighted eigenvalue problems of the form \eqref{eq_Sturm_Liouville}, set on $(\omega, p)$.
\begin{lemma}
    \label{theo:extension_Kroger}
    Let $d\geq 2$, $1 \leq j < d$ and let $\omega \subseteq \R^j$ be an open, bounded and convex set and $p \colon \omega \to \R_+$ a function $\frac{1}{d-j}$-concave. Then
    \begin{equation}
        \label{eq:per_lemma_1D_function}
        \mu_k(\omega,p) \leq \frac{\mu_{k,d}^*}{D_{\omega}^2}.
    \end{equation}
   Moreover equality occurs only when $j = 1$ and $p = p_{k,d}^*$.
\end{lemma}

\begin{proof}
    If  $j=1$ there is nothing to prove.  In this case $p$ is a $\frac{1}{d-1}$-concave function and so
    \[
        \mu_k \bigl( [0,D_{\omega}],p \bigr) \leq \frac{\mu_{k,d}^*}{D_{\omega}^2},
    \]
    with equality  if and only if $p = p_{k,d}^*$.

    Let us assume $j \geq 2$. First, we prove that inequality \eqref{eq:per_lemma_1D_function} holds.
The diameter of $\omega$ being placed along the $x_1$ axis, we  decompose $x \in \R^j$ as $x=(x_1,x') \in \R \times \R^{j-1}$ and  define

    \begin{align*}
      \forall x_1, \;\;  S_{x_1}(\omega) &= \bigl \{ (x_1,x') \colon x' \in \R^{j-1}, (x_1,x') \in \omega \bigr\} \\
        \forall x', \;\;  S_{x'}(\omega) &= \bigl \{ (x_1,x') \colon  x_1\in \R,  (x_1,x') \in \omega \bigr\}.
    \end{align*}
We set
    \[
         \tilde{p}(x_1) = \int_{S_{x_1}(\omega)} p(x_1,x') \, dx' \qquad \forall x_1 \in [0,D_{\omega}].
    \]
    Then  $\tilde p_1$ is $\frac{1}{d-1}$ concave  and
    \begin{equation}
        \label{eq:prima_stima_inequality}
        \mu_k(\omega,p) \leq \mu_k\bigl( [0,D_{\omega}], \tilde{p} \bigr) \leq \mu_k\bigl( [0,D_{\omega}],p_{k,d}^* \bigr) = \frac{\mu_{k,d}^*}{D_{\omega}^2},
    \end{equation}
   where the first inequality follows from the variational characterization of $\mu_k(\omega,p)$ by choosing test functions that depend only on the first variable, while the second follows from Theorem \ref{theo:Henrot_Michetti}. 
    
We analyse the equality case. Assume for contradiction
    \[
        \mu_k(\omega,p) = \frac{\mu_{k,d}^*}{D_{\omega}^2}. 
    \]
  As a consequence of  \eqref{eq:prima_stima_inequality}, we have
    \[
        \tilde{p} = p_{k,d}^*,
    \]
    and from the fact that $p_{k,d}^*$ is piecewise affine, the equality case of Brunn-Minkowski, Theorem \ref{theo:Brunn_Minkowski}, implies that the set
    \[
        \Omega = \{ (a,b) \in \R^j \times \R^{d-j} \colon a \in \omega, \norma{b} < p(a)^{\frac{1}{d-j}} \},
    \]
 has a piecewise conical structure. Moreover the eigenfunction $u_k$ of $\mu_k\bigl( [0,D_{\omega}],\tilde{p} \bigr)$ is also eigenfunction of single variable $x_1$ for $(\omega,p)$.

    We use the fact that $u = u_k(x_1)$ is the $k$-th eigenfunction associated with $(\omega,p)$ and test the weak formulation of the equation with functions of the form $\varphi(x_1)\psi(x')$, where $x' \in S_{x_1}(\omega)$.
    
Let us first consider $\psi(x')$. Since $x_1$ and $x'$ are separated variables, the gradients of $u_k$ and $\psi$ are pointwise orthogonal, so that
    \[
        0 = \int_{\omega} \nabla u_k \nabla \psi  p(x_1,x')dx =  \mu_{k,d}^* \int_{\omega} u_k(x_1) \psi(x') p(x_1,x') \, dx.
    \]
 We can consider a sequence $\psi_{\eps}$ of functions converging in $L^2$ to $\frac{1}{\omega_{j-1} r^{j-1}} \1_{B'(x',r)}$ where $B'(x',r)$ is the ball in $\R^{j-1}$ centered at $x'$ and with radius $r$, hence passing to the limit for $\eps \to 0$, we get
    \[
        \frac{1}{\omega_{j-1} r^{j-1}}\int_{\omega} u_k(x_1) \1_{B'(x',r)} p(x_1,x') \, dx = 0,
    \]
    and then, passing to the limit for $r \to 0$, we get for all $x'$
    \begin{equation}
        \label{eq:integral_S_x'_equal_0}
        \int_{S_{x'}(\omega)} u_k(x_1) p(x_1,x') \, dx_1 = 0.
    \end{equation}
    We proceed in the same way with the test function $x\to \varphi(x_1)\psi_{\eps}(x')$ where $\psi_{\eps} \to \frac{1}{\omega_{j-1} r^{j-1}}\1_{B'(x',r)}$ and, as before, we get
    \[
        \int_{S_{x'}(\omega)} u_k'(x_1) \varphi'(x_1)p(x_1,x') \, dx_1 = \mu_{k,d}^* \int_{S_{x'}(\omega)} u_k(x_1)\varphi(x_1)p(x_1,x') \, dx_1.
    \]
    If $\varphi \in C_C^{\infty}(S_{x'}\bigl( \omega) \bigr)$ we can integrate by parts (relying on the   structure  of $p$ from Theorem \ref{theo:Henrot_Michetti} and the inner smoothness of $u$) obtaining
    \begin{equation}
        \label{eq:intermediate_for_extension_Kroger}
        -u_k''(x_1)-u_k'(x_1) \frac{p'(x_1,x')}{p(x_1,x')} = \mu_{k,d}^* u_k(x_1) \qquad \text{ in } S_{x'}(\omega).
    \end{equation}
    Therefore the function
    \[
        \frac{p'(x_1,x')}{p(x_1,x')},
    \]
    does not depend on $x'$ since $u$ itself does not depend on $x'$ in \eqref{eq:intermediate_for_extension_Kroger}.  Assume that $\frac{p'(x_1,x')}{p(x_1,x')}= f(x_1)$ for every $x_1 \in (0, D_\omega)$.  As a consequence, for fixed $x'$ and $0< a<b<D_\omega$ such that $(a,x'), (b,x') \in S_{x'}(\omega)\setminus \partial \omega$ we have
    $$ \ln p(b, x')-\ln p(a,x')= \int _a^b f(s) ds.$$
    In particular,
    $$
     p(b, x')=p(a,x') e^{\int_a^b f(s)ds}.$$
     If $(x_1,x') \in \partial \omega$ is such that $p(x_1,x') =0$, then $p$ vanishes on $ S_{x'}(\omega)$, fact  which is not possible from the structure of $p$.

    The weak form of the equation and the fact that $u_k$ is smooth on $(0, D_\omega)$ give
    \[
        u_k'(x_1) p(x_1,x')= 0 \qquad \text{ on } \partial S_{x'}(\omega),
    \]
 hence  $u_k'(x_1) = 0$ at any point of  $S_{x'}(\omega)$ where $p(x_1,x') \not =0$. Since at any boundary point $(a, x') \in \partial \omega$ the function  $p$ can not vanish, we conclude that $u_k'(a)=0$, meaning that $u_k$ vanishes on the interval $(0, D_\omega)$.    
    
    \end{proof}

\section{Proof of Theorems \ref{main_theorem_k=2} and \ref{teo:optimality_exponent_2} }
\label{sec:proof_of_theorems}
Now we are in position to prove Theorem \ref{main_theorem_k=2}.  For clarity of exposition and to make the proof easier to follow, we begin by treating the case 
$k=1$. The argument for general 
$ k $ proceeds similarly, with only additional technical details.

\begin{proof}[Proof of Theorem \ref{main_theorem_k=2} for $k = 1$.]
 Let us argue by contradiction: assume that there exists $\Set{ \Omega_n }_{n \in \N}$ a sequence of open convex sets with $D_{\Omega_n} = 1$ such that
    \begin{equation}
        \label{eq:contradiction_hp_k=1}
        \mu_1(\Omega_n) \geq \mu_{1,d}^* - \frac{1}{n} a_2(\Omega_n)^2.
    \end{equation}
Up to a subsequence, it  converges in the sense of Hausdorff to $\Omega \subseteq \R^d$, a  convex set of diameter $1$. 
    
    We claim that $\Omega$ is a segment.  First we prove that $\Omega_n$ should collapse and, second,  that it collapses into a segment.

    \begin{itemize}
        \item Assume first that $\Omega$  has an interior point (hence no collapsing). Therefore, we  have
        \[
            \mu_1(\Omega_n) \to \mu_1(\Omega) = \mu_{1,d}^*,
        \]
        which   is in contradiction with  the equality case in Theorem \ref{theo:Henrot_Michetti}. Consequently,  $\{ \Omega_n \}$ collapses.
        \item Now let us show that $\{ \Omega_n \}$ converges to a segment. Assume for contradiction that for some $j \ge 2$, we can write  $\Omega =\omega\times \{0\}^{d-j}$ where $\omega\subset \R^j$ has an interior point.  Then  
        \[
          \mu_1(\Omega_n) \to   \mu_1(\omega,p) = \mu_{1,d}^*,
        \]
        for a suitable $\frac{1}{d-j}$-concave function $p$. This  is in contradiction to the equality case in Theorem \ref{theo:extension_Kroger}.    
    \end{itemize}
    
      Therefore $j=1$, $\omega$ is a segment and, since $D_{\Omega_n} = 1$ we can assume that $\omega = (0,1)\times\{0\}^{d-1}$  and all $\Omega_n$ have $\omega$ as a diameter.  
Setting $\tilde{\Omega}_n \coloneqq T_n(\Omega_n) $ there exists $\tilde{\Omega}$ such that $T_n(\Omega_n) \to \tilde{\Omega}$, up to a subsequence. We   define
    \[
        p(x_1) = \Ho^{d-1}(S_{x_1} \cap \tilde{\Omega}).
    \]
    Since $\mu_1(\omega,p) = \mu_{1,d}^*$, we have
    \[
        p(x)^{\frac{1}{d-1}} = \bigl( p_{1,d}^* \bigr)^{\frac{1}{d-1}} =
        \begin{cases}
            2x_1 & \text{ if } \displaystyle{ x \in \biggl[ 0,\frac{1}{2} \biggr]} \\[2ex]
            2(1-x_1) & \text{ if } \displaystyle{ x \in \biggl[ 0,\frac{1}{2} \biggr]}.
        \end{cases}
    \]
    This implies that
    \[
        \tilde{\Omega} \cap S_{x_1} =
        \begin{cases}
            x_1 \cdot S & \text{ on } \displaystyle{ x_1 \in \biggl[ 0,\frac{1}{2} \biggr]} \\[2ex]
            (1-x_1) \cdot S & \text{ on } \displaystyle{x_1 \in \biggl[ \frac{1}{2},1 \biggr]},
        \end{cases}
    \]
    where $S$ is a convex set in $\{ 0 \} \times \R^{d-1}$ containing the origin (Figure \ref{fig:cone_proof_k=1}).

    \begin{figure}[h!]
        \centering
        \begin{tikzpicture}
            \tikzmath{
            \y = 0.4;
            \h = 1.2;
            }
            \draw[black] (0,0) -- (6,0);
            \filldraw[black] (0,0) circle (1pt) node[anchor = north] {$0$};
            \filldraw[black] (6,0) circle (1pt) node[anchor = north] {$1$};

            \filldraw[black] (3,0) circle (1pt) node[anchor = south] {$\frac{1}{2}$};

            \draw[blue] (3,\y) ellipse(0.3 and \h);

            \node[blue] at(3,{\y- 0.5 + \h}) {$S$};

            \draw[black] (0,0) -- (3,{\y+\h}) -- (6,0);
            \draw[black] (0,0) -- (3,{\y-\h}) -- (6,0);

            \node[black] at (4.7,1.2) {$\tilde{\Omega}$};
            
        \end{tikzpicture}
        \caption{}
        \label{fig:cone_proof_k=1}
    \end{figure}

    To get a contradiction, we use the weak formulation of $\mu_1(\Omega_n)$ and test it with a suitable test function. Let us consider
\begin{equation}\label{d01}
        w_n(x_1,x') = u_n(x_1)\Bigl[ 1+\tau  \bigl( x_2-c(x_1) \bigr)^2  a_2(\Omega_n)^2\Bigr] - d_n,  \qquad x_1\in (0,1), x' \in S_{x_1} \cap \Omega_n,
\end{equation}
    where $u_n$ is the first eigenfunction of the following Sturm-Liouville problem
    \[
        \begin{cases}
            \displaystyle{ - \frac{d}{dx_1} \bigl( u_n'(x_1) p_n(x_1) \bigr) = \mu_{1,n}u_n(x_1) p_n(x_1)} & \text{ for } x \in (0,1) \\[1.6ex]
            u_n'(x_1) p_n(x_1) = 0 & \text{ on } \{ 0,1 \},
        \end{cases}
        \qquad
        p_n(x_1) = \Ho^{d-1} (S_{x_1} \cap \Omega_n),
    \]
    normalized  by
    \begin{equation}
        \label{eq:normalization_u_n}
        \int_{\Omega_n} u_n(x_1)^2 \, dx = \int_0^1 u_n(x_1)^2 p_n(x_1) \, dx_1 = 1.
    \end{equation}
The value of $c$ is defined as in the previous section,   by
    \[
        c(x_1) = 
        \dashint_{S_{x_1} \cap \tilde{\Omega}} x_2 \, d\Ho^{d-1}.
    \]
In \eqref{d01},  $d_n$ is a constant such that $ w_n$ has zero average and $\tau$ will be chosen later.

    For $(x_1,\ldots,x_d) \in \tilde{\Omega}_n = T_n(\Omega_n)$ we define
    \begin{equation}
        \label{eq:def_tilde_w_n}
        \begin{split}
            \tilde{w}_n(x_1,x')
            & =
            k_n w_n \bigl( x_1, a_2(\Omega_n) x_2, \ldots, a_d(\Omega_n) x_d \bigr)
            =
            \tilde{u}_n(x_1) \Bigl[ 1+ \tau \bigl( x_2-c(x_1) \bigr)^2 a_2(\Omega_n)^2 \Bigr] - k_nd_n \\
            \tilde{p}_n(x_1) &= \frac{p_n(x_1)}{k_n ^2},
        \end{split}
    \end{equation}
    where
    \[
        \tilde{u}_n(x_1) = k_n u_n(x_1), \qquad k_n = \sqrt{\prod_{j=2}^d a_j(\Omega_n)}.
    \]
    Hence,  we can write
    \[
        \frac{\displaystyle \int_{\Omega_n} \bigl \lvert \nabla w_n \bigr \rvert^2 \, dx }{\displaystyle \int_{\Omega_n} w_n^2 \, dx }
        =
        \frac{\displaystyle \int_{\tilde{\Omega}_n} \biggl( \parzder{\tilde{w}_n}{x_1} \biggr)^2 + \sum_{i=2}^d \biggl( \frac{1}{a_i(\Omega_n)^2} \parzder{\tilde{w}_n}{x_i} \biggr)^2 \, dx }{\displaystyle \int_{\tilde{\Omega}_n} \tilde{w}_n^2 \, dx }
        =
        \frac{\displaystyle \int_{\tilde{\Omega}_n} \biggl( \parzder{\tilde{w}_n}{x_1} \biggr)^2 + \frac{1}{a_2(\Omega_n)^2} \biggl( \parzder{\tilde{w}_n}{x_2} \biggr)^2 \, dx }{\displaystyle \int_{\tilde{\Omega}_n} \tilde{w}_n^2 \, dx },
    \]
    where for the last equality sign we used the fact that $w_n$ depends only on the first two variables.
    
    Taking into account \eqref{eq:def_tilde_w_n}, equation \eqref{eq:normalization_u_n} becomes
    \[
        \int_{\tilde{\Omega}_n} \tilde{w}_n(x_1)^2 \, dx = \int_0^1 \tilde{u}_n(x_1)^2 \tilde{p}_n(x_1) \, dx_1 = 1
    \]
    Let us compute $d_n$:
    \begin{equation}
        \label{eq:definition_c_n}
        \begin{split}
            \int_{\Omega_n} w_n \, dx = 0
            & \iff \int_{\Omega_n} \Bigl[ u_n(x_1) \Bigl( 1+\tau   \bigl( x_2-c(x_1) \bigr)^2  a_2(\Omega_n)^2 \Bigr) - d_n \Bigr] \, dx_1 = 0 \\
            & \iff d_n = \frac{\tau}{\abs{\Omega_n}} \int_{\Omega_n} u_n(x_1)  \bigl( x_2-c(x_1) \bigr)^2  a_2(\Omega_n)^2 \, dx \\
            & \iff d_n = \frac{\tau}{\lvert \tilde{\Omega}_n \rvert} \Biggl( \int_{\tilde{\Omega}_n} \tilde{u}_n(x_1) \bigl( x_2- c(x_1) \bigr)^2 \, dx \Biggr) a_2(\Omega_n)^2,
        \end{split}
    \end{equation}
    since $u_n$ has zero average on $\Omega_n$.

Since $\tilde{\Omega}_n \xrightarrow{\mathcal{H}} \tilde{\Omega}$
    \[
        \bigl( \tilde{p}_n \bigr)^{\frac{1}{d-1}} \to  (p_{1,d}^*)^{\frac{1}{d-1}},
    \]
    up to a multiplicative constant. Moreover,
      $1_{\tilde \Omega_n} \tilde u_n \to 1_{\tilde \Omega} \tilde u$ and $1_{\tilde \Omega_n} \nabla \tilde u_n \to 1_{\tilde \Omega} \nabla \tilde u$ strongly in $L^2$, where $u$ is the one dimensional eigenfunction associated to $p_{1,d}^*$ on $(0,1)$.

    So the weak formulation of $\mu_1(\Omega_n)$ gives
    \begin{equation}
        \label{eq:almost_contradiction_hp_k=1}
        \mu_{1,d}^* - \frac{1}{n} a_2(\Omega_n)^2 \leq \mu_1(\Omega_n) \leq \frac{\displaystyle \int_{\Omega_n} \abs{\nabla w_n}^2 \, dx }{\displaystyle \int_{\Omega_n} w_n^2 \, dx}
        =
        \frac{\displaystyle \int_{\tilde{\Omega}_n} \biggl( \parzder{\tilde{w}_n}{x_1} \biggr)^2 + \frac{1}{a_2(\Omega_n)^2} \biggl( \parzder{\tilde{w}_n}{x_2} \biggr)^2 \, dx }{\displaystyle \int_{\tilde{\Omega}_n} \tilde{w}_n^2 \, dx }.
    \end{equation}
   Since the left-hand side contains a term of order two in $a_2(\Omega_n)$, we expand the right-hand side in powers of $a_2(\Omega_n)$ up to second order. We analyze the numerator and denominator of \eqref{eq:almost_contradiction_hp_k=1} separately.
    
    \smallskip
    \noindent{\bf 
 Estimate of the numerator.}       To simplify notations, if $v$ is a function of one variable  $x_1$ and $\tilde v$ is the extension of $v$ in $\R^d$ constant in $(x_2, \dots, x_d)$ we make no difference between $\frac{\partial \tilde v}{\partial x_1}(x)$, $\tilde v'(x_1)$ and $v'(x_1)$.
        First of all we notice that 
        \begin{align*}
            \parzder{\tilde{w}_n}{x_1}
            & =
            \tilde{u}_n'(x_1) \Bigl[ 1+\tau \bigl( x_2-c(x_1) \bigr)^2 \, a_2(\Omega_n)^2 \Bigr] - 2 \tau \tilde{u}_n(x_1) \bigl( x_2- c(x_1) \bigr) c'(x_1) a_2(\Omega_n)^2 \\
            \parzder{\tilde{w}_n}{x_2}
            & =
            2 \tau \tilde{u}_n(x_1) \bigl( x_2-c(x_1) \bigr) a_2(\Omega_n)^2.
        \end{align*}
        so,
        \begin{align*}
            \biggl( \parzder{\tilde{w}_n}{x_1} \biggr)^2
            & =
            \bigl( \tilde{u}_n'(x_1) \bigr)^2 \Bigl[ 1+\tau \bigl( x_2-c(x_1) \bigr)^2 a_2(\Omega_n)^2 \Bigr]^2
            \\
            & +
            4 \tau^2 \bigl( \tilde{u}_n(x_1) \bigr)^2 \bigl( x_2-c(x_1) \bigr)^2 \, c'(x_1)^2 \, a_2(\Omega_n)^4 \\
            & -
            4 \tau \tilde{u}_n'(x_1) \tilde{u}_n(x_1) \Bigl[ 1+\tau \bigl( x_2-c(x_1) \bigr)^" a_2(\Omega_n)^2 \Bigr] \, \bigl( x_2-c(x_1) \bigr) c'(x_1) a_2(\Omega_n)^2 \\
            \biggl( \parzder{\tilde{w}_n}{x_2} \biggr)^2
            & =
            4 \tau^2 \bigl( \tilde{u}_n(x_1) \bigr)^2 \, \bigl( x_2-c(x_1) \bigr)^2 \, a_2(\Omega_n)^4,
        \end{align*}
        therefore we have
        \begin{align*}
            \int_{\tilde{\Omega}_n} \biggl( \parzder{\tilde{w}_n}{x_1} \biggr)^2 + \frac{1}{a_2(\Omega_n)^2} \biggl( \parzder{\tilde{w}_n}{x_2} \biggr)^2 \, dx
            & =
            \int_{\Omega_n} \bigl( \tilde{u}_n'(x_1) \bigr)^2 \, dx \\
            & +
            2 \tau \Biggl( \int_{\tilde{\Omega}_n} \bigl(\tilde{u}_n'(x_1) \bigr)^2 \bigl( x_2-c(x_1) \bigr)^2 \, dx \Biggr) a_2(\Omega_n)^2 \\ 
            & -
            4 \tau \Biggl( \int_{\tilde{\Omega}_n} \tilde{u}_n'(x_1) \tilde{u}_n(x_1) \bigl( x_2-c(x_1) c'(x_1) \, dx \bigr) \Biggr) a_2(\Omega_n)^2 \\
            & +
            4 \tau^2 \Biggl( \int_{\tilde{\Omega}_n} \tilde{u}_n(x_1)^2 \bigl( x_2-c(x_1) \bigr)^2 \, dx \Biggr) a_2(\Omega_n)^2 + o \bigl( a_2(\Omega_n)^2 \bigr) \\
            & =
            \mu_1(I,p_n) +
            2 \tau \Biggl( \int_{\tilde{\Omega}_n} \tilde{u}_n'(x_1)^2 \bigl( x_2-c(x_1) \bigr)^2 \, dx \Biggr) a_2(\Omega_n)^2 \\ 
            & -
            4 \tau \Biggl( \int_{\tilde{\Omega}_n} \tilde{u}_n'(x_1) \tilde{u}_n(x_1) \bigl( x_2-c(x_1) c'(x_1) \, dx \bigr) \Biggr) a_2(\Omega_n)^2 \\
            & +
            4 \tau^2 \Biggl( \int_{\tilde{\Omega}_n} \tilde{u}_n(x_1)^2 \bigl( x_2-c(x_1) \bigr)^2 \, dx \Biggr) a_2(\Omega_n)^2 + o \bigl( a_2(\Omega_n)^2 \bigr).
        \end{align*}
        Recalling that $\mu_1(I,p_n) \leq \mu_{1,d}^*$, we get
        \begin{equation}
            \label{eq:estimate_numerator_k=1}
            \begin{split}
                \int_{\tilde{\Omega}_n} \biggl( \parzder{\tilde{w}_n}{x_1} \biggr)^2 + \frac{1}{a_2(\Omega_n)^2} \biggl( \parzder{\tilde{w}_n}{x_2} \biggr)^2 \, dx
                & \leq
                \mu_{1,d}^* + 2 \tau \Biggl(\int_{\tilde{\Omega}_n} \bigl( \tilde{u}_n'(x_1) \bigr)^2 \, \bigl( x_2 - c(x_1) \bigr)^2 \, dx \Biggr) a_2(\Omega_n)^2 \\
                & -
                4 \tau \Biggl( \int_{\tilde{\Omega}_n} \tilde{u}_n'(x_1)\tilde{u}_n(x_1) \bigl( x_2-c(x_1) \bigr) \bigl( c(x_1) \bigr) \, dx \Biggr) a_2(\Omega_n)^2 \\
                & +
                4 \tau^2 \Biggl( \int_{\tilde{\Omega}_n} \tilde{u}_n(x_1)^2 \, \bigl( x_2 - c(x_1) \bigr)^2 \, dx \Biggr) a_2(\Omega_n)^2 + o \bigl( a_2(\Omega_n)^2 \bigr).
            \end{split}
        \end{equation}
   
\smallskip
\noindent{\bf Estimate of the denominator.} Using the expression of $d_n$ from \eqref{eq:definition_c_n}, we have    

        \begin{equation}
            \label{eq:numerator_k=1}
            \begin{split}
                \int_{\tilde{\Omega}_n} \tilde{w}_n(x)^2 \, dx
                & =
                \int_{\Omega_n} \tilde{u}_n(x_1)^2 \Bigl[ 1+\tau \bigl( x_2 - c(x_1) \bigr)^2 a_2(\Omega_n)^2 \Bigr]^2 \, dx + k_n^2 d_n^2 \abs{\Omega_n} \\
                & -
                2k_n d_n \tau \int_{\Omega_n} \tilde{u}_n(x_1) \bigl( x_2-c(x_1) \bigr)^2  a_2(\Omega_n)^2   \, dx \\
                & =
                \int_{\Omega_n} \tilde{u}_n(x_1)^2 \Bigl[ 1+\tau \bigl( x_2 - c(x_1) \bigr)^2 a_2(\Omega_n)^2 \Bigr]^2 \, dx - k_n^2 d_n^2 \abs{\Omega_n} \\
                & =
                \int_{\tilde{\Omega}_n} \tilde{u}_n(x_1)^2 \, dx
                +
                2 \tau \Biggl( \int_{\tilde{\Omega}_n} \tilde{u}_n(x_1)^2  \bigl( x_2- c(x_1)\bigr)^2   \, dx \Biggr) a_2(\Omega_n)^2 + o \bigl( a_2(\Omega_n)^2 \bigr).
            \end{split}
        \end{equation}
 
        \smallskip

    So \eqref{eq:almost_contradiction_hp_k=1} becomes
    \begin{align*}
        \mu_{1,d}^* & - \frac{1}{n} a_2(\Omega_n)^2 \\
        & \leq
        \frac{\displaystyle \mu_{1,d}^* + 2 \tau \Biggl( \int_{\tilde{\Omega}_n} \bigl( \tilde{u}_n'(x_1) \bigr)^2 \, \bigl( x_2 - c(x_1) \bigr)^2 \, dx \Biggr) a_2(\Omega_n)^2}{\displaystyle 1 + 2 \tau \Biggl( \int_{\tilde{\Omega}_n} \tilde{u}_n(x_1)^2 \bigl( x_2 - c(x_1) \bigr)^2 \, dx \Biggr) a_2(\Omega_n)^2 + o \bigl( a_2(\Omega_n)^2 \bigr) } \\
        & -
        \frac{\displaystyle 4 \tau \Biggl( \int_{\tilde{\Omega}_n} \tilde{u}_n'(x_1)\tilde{u}_n(x_1) \bigl( x_2-c(x_1) \bigr) c'(x_1) \, dx \Biggr) a_2(\Omega_n)^2 }{ \displaystyle 1 + 2 \tau \Biggl( \int_{\tilde{\Omega}_n} \tilde{u}_n(x_1)^2 \bigl( x_2 - c(x_1) \bigr)^2 \, dx \Biggr) a_2(\Omega_n)^2 + o \bigl( a_2(\Omega_n)^2 \bigr) }
        \\
        & +
        \frac{\displaystyle 4 \tau^2 \Biggl( \int_{\tilde{\Omega}_n} \tilde{u}_n(x_1)^2 \, \bigl( x_2 - c(x_1) \bigr)^2 \, dx \Biggr) a_2(\Omega_n)^2 + o \bigl( a_2(\Omega_n)^2 \bigr) }{\displaystyle 1 + 2 \tau \Biggl( \int_{\tilde{\Omega}_n} \tilde{u}_n(x_1)^2 \bigl( x_2 - c(x_1) \bigr)^2 \, dx \Biggr) a_2(\Omega_n)^2 + o \bigl( a_2(\Omega_n)^2 \bigr)}.
    \end{align*}
    Therefore we get
    \begin{equation}
        \label{eq:label_da_richiamare}
        \begin{split}
            -\frac{1}{n}a_2(\Omega_n)^2 + o\bigl( a_2(\Omega_n)^2 \bigr)
            & \leq
            2 \tau \Biggl[ \int_{\tilde{\Omega}_n} \Bigl( \bigl( \tilde{u}_n'(x_1) \bigr)^2 - \mu_{1,d}^* \tilde{u}_n(x_1)^2 \Bigr) \, \bigl( x_2 - c(x_1) \bigr)^2 \, dx  \\
            & -
            2 \int_{\tilde{\Omega}_n} \tilde{u}_n'(x_1)\tilde{u}_n(x_1) \bigl( x_2-c(x_1) \bigr) c'(x_1) \, dx \Biggr] a_2(\Omega_n)^2 \\
            & +
            4 \tau^2 \Biggl( \int_{\tilde{\Omega}_n} \tilde{u}_n(x_1)^2 \, \bigl( x_2 - c(x_1) \bigr)^2 \, dx \Biggr) a_2(\Omega_n)^2 \\
            & \eqqcolon
            2 \gamma_n a_2(\Omega_n)^2 \tau + 4 \beta_n a_2(\Omega_n)^2 \tau^2,
        \end{split}
    \end{equation}
    where
    \begin{align*}
        \gamma_n & = \int_{\tilde{\Omega}_n} \Bigl( \bigl( \tilde{u}_n'(x_1) \bigr)^2 - \mu_{1,d}^* \tilde{u}_n(x_1)^2 \Bigr) \, \bigl( x_2 - c(x_1) \bigr)^2 \, dx
        -
        2 \int_{\tilde{\Omega}_n} \tilde{u}_n'(x_1)\tilde{u}_n(x_1) \bigl( x_2-c(x_1) \bigr) c'(x_1) \, dx, \\
        \beta_n &= \int_{\tilde{\Omega}_n} \tilde{u}_n(x_1)^2 \, \bigl(x_2 - c(x_1) \bigr)^2 \, dx.
    \end{align*}
  We observe that the right-hand side of \eqref{eq:label_da_richiamare} contains a term of order two in $a_2(\Omega_n)$, whose coefficient is a second-order polynomial in $\tau$ with no constant term. A contradiction is then obtained by choosing $\tau < 0$ sufficiently small, provided that    \[
        \lim_{n \to +\infty} \gamma_n \neq 0.
    \]
  This means
    \begin{align*}
        \lim_{n \to + \infty} \Biggl \{ \int_{\tilde{\Omega}_n} \Bigl[ \bigl( \tilde{u}_n'(x_1) \bigr)^2  - \mu_{1,d}^* \tilde{u}_n(x_1)^2 \Bigr] \, \bigl( x_2 - c(x_1) \bigr)^2 \, dx - 2 \int_{\tilde{\Omega}_n} \tilde{u}_n'(x_1)\tilde{u}_n(x_1) \bigl( x_2-c(x_1) \bigr) c'(x_1) \, dx \Biggr \}
        \neq 0.
    \end{align*}
   
    We can rewrite each of the two terms above
        $$
        \int_{\tilde{\Omega}_n} \Bigl[ \bigl( \tilde{u}_n'(x_1) \bigr)^2 - \mu_{1,d}^* \tilde{u}_n(x_1)^2 \Bigr] \bigl( x_2  - c(x_1) \bigr)^2  \, dx \hskip 3cm  $$ 
        $$     \hskip 3cm    =
        \int_0^1 \Bigl[ \bigl( \tilde{u}_n'(x_1) \bigr)^2 - \mu_{1,d}^* \tilde{u}_n(x_1)^2 \Bigr] \biggl( \int_{S_{x_1} \cap \, \tilde{\Omega}_n} \bigl( x_2 - c(x_1) \bigr)^2 \, d\Ho^{d-1} \biggr) \, dx_1,
$$
    and
    \begin{align*}
        \int_{\tilde{\Omega}_n} \tilde{u}_n'(x_1)\tilde{u}_n(x_1) \bigl( x_2 & -c(x_1) \bigr) c'(x_1) \, dx 
        = \int_0^1 \Bigl[ \tilde{u}_n'(x_1) \tilde{u}_n(x_1) c'(x_1) \Bigr] \biggl( \int_{S_{x_1} \cap \tilde{\Omega}_n} x_2 - c(x_1) \, d \Ho^{d-1}  \biggr) \, dx_1.
    \end{align*}
    We recall that $\tilde{\Omega}_n$ converges to $\tilde{\Omega}$ in the sense of Hausdorff and $\{ 1_{\tilde \Omega_n}\tilde{u}_n \}, \{ 1_{\tilde \Omega_n}\nabla \tilde{u}_n \}$ satisfy (Theorem \ref{theo:ABF_24} and Remark \ref{d31})
        \begin{align*}
        1_{\tilde \Omega_n} \tilde{u}_n & \to 1_{\tilde \Omega}\tilde{u} & &\text{ in } L^2  \\
       1_{\tilde \Omega_n} \nabla \tilde{u}_n & \to 1_{\tilde \Omega} \nabla \tilde{u} & &\text{strongly  in } L^2,
    \end{align*}
    hence we have
    \begin{align*}
        \lim_{n \to +\infty } \int_0^1 \Bigl[ \bigl( \tilde{u}_n'(x_1) \bigr)^2 & - \mu_{1,d}^* \tilde{u}_n(x_1)^2 \Bigr] \biggl( \int_{S_{x_1} \cap \, \tilde{\Omega}_n} \bigl( x_2 - c(x_1) \bigr)^2 \, d\Ho^{d-1} \biggr) \, dx_1 \\
        &    = \lim_{n \to +\infty } \int_{\tilde \Omega_n} \Bigl[ \bigl( \frac{\partial \tilde{u}_n}{\partial x_1} \bigr)^2  - \mu_{1,d}^* \tilde{u}_n(x_1)^2 \Bigr] \biggl( \dashint_{S_{x_1} \cap \, \tilde{\Omega}_n} \bigl( x_2 - c(x_1) \bigr)^2 \, d\Ho^{d-1} \biggr) \, dx_1   \\
         &    =   \int_{\tilde \Omega} \Bigl[ \bigl( \frac{\partial \tilde{u}}{\partial x_1} \bigr)^2  - \mu_{1,d}^* \tilde{u}^2 \Bigr] \biggl( \dashint_{S_{x_1} \cap \, \tilde{\Omega}} \bigl( x_2 - c(x_1) \bigr)^2 \, d\Ho^{d-1} \biggr) \, dx   \\
        & =
        \int_0^1 \Bigl[ \bigl( \tilde{u}'(x_1) \bigr)^2 - \mu_{1,d}^* \tilde{u}^2 \Bigr] \biggl( \int_{S_{x_1} \cap \, \tilde{\Omega}} \bigl( x_2 - c(x_1) \bigr)^2 \, d\Ho^{d-1} \biggr) \, dx.
    \end{align*}
 From the definition of $c$, we get
    \begin{align*}
        \lim_{n \to +\infty } & \int_0^1 \Bigl[ \tilde{u}_n'(x_1) \tilde{u}_n(x_1) c'(x_1) \Bigr] \biggl( \int_{S_{x_1} \cap \tilde{\Omega}_n} x_2 - c(x_1) \, d \Ho^{d-1}  \biggr) \, dx_1 \\
        & =
        \int_{0}^1 \tilde{u}'(x_1) \tilde{u}(x_1)c'(x_1) \Biggl[ \int_{S_{x_1} \cap \tilde{\Omega}} x_2-c(x_1) \, d\Ho^{d-1} \Biggr]\ dx_1=0.
    \end{align*}
   For the reader's convenience, and to simplify the notation, we identify $\tilde{u}$ with $u$. Recall that, for every $d$, we have
    \[
        p_{\tilde{\Omega}}^{\frac{1}{d-1}} (x_1) =
        \begin{cases}
            2x_1 & \text{ in } \displaystyle{\biggl[ 0,\frac{1}{2} \biggr]} \\[1.8ex]
            2(1-x_1) & \text{ in } \displaystyle{\biggl[ \frac{1}{2},1 \biggr]}
        \end{cases}
        \quad
        \text{ and }
        \quad
        u (x_1) = -u ( 1-x_1 ), u'(x_1) = u'(1-x_1) \qquad \forall x_1 \in \biggl[ 0,\frac{1}{2} \biggr].
    \]
Using \eqref{eq:momento_di_ordine_2}, there exists $K>0$ such that
    \begin{equation}
        \label{eq:giusta_da_richiamare}
        \int_0^1 \Bigl[ \bigl( u'(x_1) \bigr)^2 - \mu_{1,d}^* u(x_1)^2 \Bigr] \biggl( \int_{S_{x_1} \cap \, \tilde{\Omega}} \bigl( x_2 - c(x_1) \bigr)^2 \, d\Ho^{d-1} \biggr) \, dx_1
        \hskip 2cm
        \end{equation}
        \[\hskip 2cm =
        2 K \int_0^{\frac{1}{2}} \Bigl[ \bigl( u'(x_1) \bigr)^2 - \mu_{1,d}^* u(x_1)^2 \Bigr] p(x_1)^{1+\frac{2}{d-1}} \, dx_1.
  \]
    Since $u$ is an eigenfunction of the Sturm-Liouville problem related to $p(x_1)$ and taking into account the regularity of $u$ and $p$ (see \cite{Henrot_Michetti}), it holds
    \begin{equation}
         \label{eq:integration_by_parts_v}
        -u''(x_1)p(x_1) - \mu_{1,d}^* u(x_1) p(x_1) = u'(x_1)p'(x_1)
        \qquad x_1 \in \biggl( 0,\frac{1}{2} \biggr).
    \end{equation}
Multiplying  by $u(x_1) p(x_1)^\frac{2}{d-1}$, integrating from $0$ to $\frac{1}{2}$      and using $p(0)=0$ and $u(\frac 12)=0$, we get (see also   \cite[Lemma 3.10]{Henrot_Michetti})

    \[
        \int_{0}^{\frac{1}{2}} \Bigl[ u'(x_1)^2 - \mu_{1,d}^* u(x_1)^2 \Bigr] \,p(x_1)^{1+\frac{2}{d-1}} \, dx_1
        =
        - \frac{2}{d-1} \int_0^{\frac{1}{2}} u'(x_1)u(x_1) p(x_1)^{\frac{2}{d-1}} \, p'(x_1) \, dx_1.
    \]
        Plugging this into \eqref{eq:giusta_da_richiamare}, we get
    \begin{equation}
        \label{eq:last_integration_by_parts_k=1}
        \int_0^1 \Bigl[ \bigl( u'(x_1) \bigr)^2 - \mu_{1,d}^* u(x_1)^2 \Bigr] \biggl( \int_{S_{x_1} \cap \, \tilde{\Omega}} \bigl( x_2 - c(x_1) \bigr)^2 \, d\Ho^{d-1} \biggr) \, dx_1 \hskip 2cm
         \end{equation}
    \[ \hskip 2cm   =
        -\frac{4K}{d-1} \int_0^{\frac{1}{2}} u'(x_1)u(x_1) \, p(x_1)^{\frac{2}{d-1}} p'(x_1) \, dx_1.\]
   Let us show that the last term is different from zero. Integrating by parts   and using again $p(0)=0$ and $u(\frac 12)=0$, we get  
    \begin{equation*}
        \int_0^{\frac{1}{2}} u'(x_1)u(x_1) \, p(x_1)^{\frac{2}{d-1}} p'(x_1) \, dx_1
        =
        -\int_0^{\frac{1}{2}} \frac{u(x_1)^2}{2} \frac{d}{dx_1} \Bigl( p(x_1)^{\frac{2}{d-1}} p'(x_1) \Bigr) \, dx_1.
    \end{equation*}
Since $p^{\frac{1}{d-1}}$ is affine on $\bigl[ 0, \frac{1}{2} \bigr]$,  there exists $\tilde{a}>0$ such that
    \begin{align*}
        \frac{d}{dx_1} \bigl( p^{\frac{1}{d-1}}(x_1) \bigr) = \tilde{a}
        & \iff
        \frac{1}{d-1} \frac{p'(x_1)}{p(x_1)^{1-\frac{1}{d-1}}} = \tilde{a} \\[0.8ex]
        & \iff
        p'(x_1) = \tilde{a} (d-1)p(x_1)^{1-\frac{1}{d-1}},
    \end{align*}
    therefore
    \begin{align*}
        \frac{d}{dx_1} \Bigl( p(x_1)^{\frac{2}{d-1}} p'(x_1) \Bigr)
        & =
        \tilde{a} (d-1) \frac{d}{d x_1} \Bigl( p(x_1)^{1+\frac{1}{d-1}} \Bigr) \\[1ex]
        & =
        \tilde{a} (d-1)\biggl( 1+\frac{1}{d-1} \biggr) p(x_1)^{\frac{1}{d-1}} p'(x_1) \\[1.3ex]
        & =
        \tilde{a}^2 d(d-1) p(x_1).
    \end{align*}
Equality \eqref{eq:last_integration_by_parts_k=1} becomes
    \[
        \int_0^1 \Bigl[ \bigl( u'(x_1) \bigr)^2 - \mu_{1,d}^* u(x_1)^2 \Bigr] \biggl( \int_{S_{x_1} \cap \, \tilde{\Omega}} \bigl( x_2 - c(x_1) \bigr)^2 \, d\Ho^{d-1} \biggr) \, dx_1
        =
        2 \tilde{a}^2 K d
        \int_0^{\frac{1}{2}} u(x_1)^2 p(x_1).
    \]
    We conclude the proof since the previous integral is strictly positive.
\end{proof}

Now we prove Theorem \ref{main_theorem_k=2} in full generality.
\begin{proof}[Proof of Theorem \ref{main_theorem_k=2} for any $k$]
    The first part of the proof is entirely analogous. We argue by contradiction and assume that there exists a sequence $\{ \Omega_n \}$ of open convex sets with diameter equal to $1$ that violates \eqref{eq_quantitative_mu_2}, namely    \begin{equation}
        \label{eq:hp_absurd_k=2}
        \mu_k(\Omega_n) > \mu_{k,d}^* - \frac{1}{n} a_2(\Omega_n)^2.
    \end{equation}
   Reasoning as in the previous case, we can show that, up to a subsequence, $\Omega_n$ converges to a segment. Up to a rotation and a translation, we may assume that $\Omega_n \to \bigl( (0,1), p \bigr)$, where $p$ is a $\frac{1}{d-1}$-concave function.
   
Moreover, using the same notation as above, we can define    \begin{align*}
        p_n(x_1) & = \Ho^{d-1} \bigl( S_{x_1} \cap \tilde{\Omega}_n \bigr) &  & \forall x_1 \in (0,1), \, \forall n \in \N \\
        p(x_1) & = \Ho^{d-1} \bigl( S_{x_1} \cap \tilde{\Omega} \bigr) &  & \forall x_1 \in (0,1),
    \end{align*}
    where $\tilde{\Omega}_n = T_n(\Omega_n)$. Up to multiple extractions, we can assume that $\tilde{\Omega}_n \overset{\Ho}{\longrightarrow} \tilde{\Omega}$, $\forall i=1, \dots, k$, $\mu_i(\Omega_n) \to \mu_{i}(I,p)$ and $\mu_i(I,p_n) \to \mu_i(I,p)$ where $I = (0,1)$. Let us highlight that
    \begin{equation}
        \label{eq:mu_k-1_<_mu_k}
        \mu_{k-1}(I,p) < \mu_{k}(I,p),
    \end{equation}
For this, we consider the first $k+1$ eigenfunctions of $(I,p_n)$ ($u_0$ is the constant function)
    \[
        u_0^n(x_1),u_1^n(x_1),\ldots,u_k^n(x_1),
    \]
    normalized by
    \begin{equation}
        \label{eq:normalization_first_k_eigenfunction}
        \int_0^1 \bigl( u_i^n(x_1)\bigr)^2 p_n(x_1) \, dx_1 = \int_{\Omega_n} \bigl( u_i^n (x_1) \bigr)^2 \, dx = 1 \qquad \forall i \in \{ 0,\ldots,k \},
    \end{equation} 
    and let $V$ be
    \[
        V = \spann \, \Bigl \{ u_0^n(x_1), u_1^n(x_1),\ldots, u_{k-1}^n(x_1), u_k^n(x_1) \Bigl[ 1+\tau  \bigl( x_2 - c(x_1) \bigr)^2a_2(\Omega_n) ^2 \Bigr] \Bigr \} \subseteq H^1(\Omega_n),
    \]
    where
    \[
        c(x_1) = 
        \dashint_{S_{x_1} \cap \tilde{\Omega}} x_2 \, d\Ho^{d-1}
        .
    \]
    
     Relyng on the variational formulation of $\mu_k$, \eqref{eq_mu_k} and using \eqref{eq:hp_absurd_k=2}  we have
    \begin{equation}
        \label{eq:contradiction_with_subspace}
        \mu_{k,d}^* - \frac{1}{n} a_2(\Omega_n)^2
        \leq
        \mu_{k}(\Omega_n)
        \leq
        \max_{v \in V} \frac{ \displaystyle \int_{\Omega_n} \abs{\nabla v}^2 \, dx }{\displaystyle \int_{\Omega_n} v^2 \, dx }
        =
        \max_{\alpha_i \in \R} \frac{ \displaystyle \int_{\Omega_n} \biggl \lvert \sum_{i=0}^{k-1} \alpha_i \nabla v_i^n +  \alpha_k \nabla v_k^n \biggr \rvert ^2 \, dx }{\displaystyle \int_{\Omega_n} \biggl( \sum_{i=0}^{k-1} \alpha_i v_i^n + \alpha_k v_k^n \biggr)^2 \, dx }
        \eqqcolon
        \mu
        ,
    \end{equation}
    where
    \begin{align*}
       v_i^n(x_1,x') & = u_i^n(x_1) & \forall i \in \{ 0, \ldots, d-1\} \\
       v_k^n(x_1,x') & = u_k^n(x_1) \Bigl[ 1 +\tau  \bigl( x_2 - c(x_1) \bigr)^2a_2(\Omega_n) ^2\Bigr].
    \end{align*}
    Taking into account the orthogonality of the eigenfunctions, i.e.
    \[
        \int_{\Omega_n} u_i^n u_j ^n \, dx = \int_{\Omega_n} \nabla u_i^n \, \nabla u_j^n \, dx = 0 \qquad \forall i \neq j,
    \]
    we can rewrite \eqref{eq:contradiction_with_subspace} as
    \begin{equation}
        \label{eq:def_xi_max}
        \mu^n = \max_{\alpha_i \in \R} \frac{
            \displaystyle \sum_{i=0}^{k-1} \alpha_i^2 \int_{\Omega_n} \abs{\nabla u_i^n}^2 \, dx + \alpha_k^2 \int_{\Omega_n} \abs{\nabla \tilde{u}_k^n}^2 \, dx + 2 \sum_{i=0}^{k-1} \alpha_i \alpha_k \int_{\Omega_n} \nabla u_i^n \, \nabla \tilde{u}_k^n \, dx
        }{
            \displaystyle \sum_{i=0}^{k-1} \alpha_i^2 \int_{\Omega_n} (u_i^n)^2 \, dx + \alpha_k^2 \int_{\Omega_n} (\tilde{u}_k^n)^2 \, dx + 2 \sum_{i=0}^{k-1} \alpha_i \alpha_k \int_{\Omega_n} u_i^n \, \tilde{u}_k^n \, dx
        }
        .
    \end{equation}
    Let us rewrite the quantity above, setting
    \begin{align*}
        \tilde{v}_i^n(x_1,x')
        & =
        k_n v_i^n \bigl( x_1, a_2(\Omega_n) x_2,\ldots, a_d(\Omega) x_d \bigr) = \tilde{u}_i^n(x_1)
        && \forall
        i \in \{ 0,\ldots, k-1 \} \\
        \tilde{v}_k^n(x_1,x')
        & =
        k_n v_k^n \bigl( x_1, a_2(\Omega_n) x_2, \ldots, a_d(\Omega_n) x_d \bigr) = \tilde{u}_k^n(x_1) \Bigl[ 1 + \tau \bigl( x_2 - c(x_1) \bigr) ^2a_2(\Omega_n)^2 \Bigr] \\
        \tilde{p}_n(x_1) & = \frac{p_n(x_1)}{ k_n^2},
    \end{align*}
    where
    \[
        \tilde{u}_i^n(x_1) = k_n u_i^n(x_1) \qquad \forall i \in \{ 0 ,\ldots, k \}
        \qquad \text{ and }\qquad
        k_n = \sqrt{\prod_{j=2}^d a_j(\Omega_n)}.
    \]
    With a change of variable, the value of $\mu^n$ in \eqref{eq:def_xi_max} is written
    \begin{equation}
        \label{eq:def_mu_max_con_riscalati}
        \small
        \begin{split}
        \max_{\alpha_i \in \R} \frac{
            \displaystyle \sum_{i=0}^{k-1} \alpha_i^2 \int_{\tilde{\Omega}_n} \bigl( (\tilde{v}_i^n) '(x_1) \bigr)^2 \, dx + \alpha_k^2 \int_{\tilde{\Omega}_n}  \Biggl[ \biggl( \parzder{\tilde{v}_k^n}{x_1} \biggr)^2 + \frac{1}{a_2(\Omega_n)^2} \biggl( \parzder{\tilde{v}_k^n}{x_2} \biggr)^2 \Biggr] \, dx + 2 \sum_{i=0}^{k-1} \alpha_i \alpha_k \int_{\tilde{\Omega}_n} (\tilde{v}_i^n)'(x_1) \, \biggl( \parzder{\tilde{v}_k^n}{x_1} \biggr) \, dx
        }{
            \displaystyle \sum_{i=0}^{k-1} \alpha_i^2 \int_{\tilde{\Omega}_n} (\tilde{v}_i^n)^2 \, dx + \alpha_k^2 \int_{\tilde{\Omega}_n} (\tilde{v}_k^n)^2 \, dx + 2 \sum_{i=0}^{k-1} \alpha_i \alpha_k \int_{\tilde{\Omega}_n} \tilde{v}_i^n \, \tilde{v}_k^n \, dx
        } 
        .
        \end{split}
    \end{equation}
    The maximum is attained at some $(\alpha_0^n, \ldots, \alpha_k^n)$ such that $\sum_{i=0}^k (\alpha_i^n)^2 = 1, \, \alpha_k^n >0$.
We want to prove that $\alpha_i^n = O\bigl( a_2(\Omega_n)^2 \bigr)$ $\forall i \in \{ 0,\ldots, k-1 \}$ while $\alpha_k = 1 - O \bigl( a_2(\Omega_n)^2\bigr)$.
    
    The optimality condition for \eqref{eq:def_mu_max_con_riscalati} reads: $ \forall i \in \{ 0, \ldots, k-1\}$  
    \begin{equation}
        \label{eq:optimality_condition_for_xi}
        2 \alpha_i^n \int_{\tilde{\Omega}_n} \bigl( (\tilde{v}_i^n)'(x_1) \bigr)^2 \, dx + 2\alpha_k^n \int_{\tilde{\Omega}_n} (\tilde{v}_i^n)'(x_1) \, \biggl( \parzder{\tilde{v}_k^n}{x_1} \biggr) \, dx
        =
        \mu \Biggl[ 2 \alpha_i^n \int_{\tilde{\Omega}_n} (\tilde{v}_i^n)^2 \, dx + 2 \alpha_k^n \int_{\tilde{\Omega}_n} \tilde{v}_i^n \, \tilde{v}_k^n \, dx \Biggr] ,
    \end{equation}
    and using that
    \begin{align*}
        \int_{\tilde{\Omega}_n} \bigl( (\tilde{v}_i^n)'(x_1) \bigr)^2 \, dx
        & =
        \int_{\tilde{\Omega}_n} \bigl( (\tilde{u}_i^n)'(x_1) \bigr)^2 \, dx =
        \mu_i(I,p_n) && \forall i \in \{ 0,\ldots,k \} \\
        \int_{\Omega_n} (\tilde{v}_i^n)'(x_1) \, \biggl( \parzder{\tilde{v}_k^n}{x_1} \biggr) \, dx
        & =
        \int_{\tilde{\Omega}_n} (\tilde{u}_1^n)'(x_1) (\tilde{u}_k^n)'(x_1) \, dx\\
        & +
        \tau \Biggl( \int_{\tilde{\Omega}_n} (\tilde{u}_i^n)'(x_1) (\tilde{u}_k^n)'(x_1) \bigl( x_2 - c(x_1) \bigr)^2 \, dx\Biggr) a_2(\Omega_n)^2 \\
        & -2 \tau \Biggl( \int_{\tilde{\Omega}_n} (\tilde{u}_i^n)'(x_1) \tilde{u}_k^n(x_1) \bigl( x_2 - c(x_1) \bigr) c'(x_1) \, dx \Biggr) a_2(\Omega_n)^2 \\
        & = \tau \Biggl( \int_{\tilde{\Omega}_n} (\tilde{u}_i^n)'(x_1) (\tilde{u}_k^n)'(x_1) \bigl( x_2 - c(x_1) \bigr)^2 \, dx\Biggr) a_2(\Omega_n)^2 \\
        & - 2 \tau \Biggl( \int_{\tilde{\Omega}_n} (\tilde{u}_i^n)'(x_1) \tilde{u}_k^n(x_1) \bigl( x_2 - c(x_1) \bigr) c'(x_1) \, dx \Biggr) a_2(\Omega_n)^2,
        & & \forall i \in \{ 0,\ldots, k-1 \} \\
        \int_{\tilde{\Omega}_n} (\tilde{v}_i^n)^2 \, dx
        & =
        \int_{\tilde{\Omega}_n} (\tilde{u}_i^n)^2 \, dx =
        1,
        & & \forall i \in \{ 0,\ldots, k \} \\
        \int_{\tilde{\Omega}_n} \tilde{v}_i ^n \tilde{v}_k^n \, dx
        & =
        \int_{\tilde{\Omega}_n} \tilde{u}_i^n(x_1) \tilde{u}_k^n(x_1) \, dx \\
        & +
        \tau \Biggl( \int_{\tilde{\Omega}_n} \tilde{u}_i^n(x_1)\tilde{u}_k^n(x_1) \bigl( x_2 - c(x_1) \bigr)^2 \, dx \Biggr) a_2(\Omega_n)^2 \\
        & = \tau \Biggl( \int_{\tilde{\Omega}_n} \tilde{u}_i^n(x_1) \tilde{u}_k^n(x_1) \bigl( x_2 - c(x_1) \bigr)^2 \, dx \Biggr) a_2(\Omega_n)^2,
        & & \forall i \in \{ 0,\ldots, k-1 \},
    \end{align*}
    we get 
    \begin{align*}
        2 \alpha_i^n 
        \bigl[ \mu^n - \mu_i(I,p_n) \bigr]
        & =
        2 \alpha_k^n \tau 
        \biggl[ \int_{\tilde{\Omega}_n} \bigl( \tilde{u}_i^n (x_1) \bigr)' \, \bigl( \tilde{u}_k^n (x_1) \bigr)' \bigl( x_2 - c(x_1) \bigr)^2 \, dx \\
        & -
       2 \int_{\tilde{\Omega}_n} \bigl( \tilde{u}_i^n (x_1) \bigr)' \bigl( \tilde{u}_k^n (x_1) \bigr) \bigl( x_2 - c(x_1) \bigr) c'(x_1) \, dx \\
        & -
        2
        \mu^n \int_{\tilde{\Omega}_n} \tilde{u}_i^n(x_1) \, u_k^n(x_1) \bigl( x_2 - c(x_1) \bigr)^2 \, dx \biggr] a_2(\Omega_n)^2.
    \end{align*}
    Let us denote with $A_n^i$ and $B_n^i$ the terms in square brakets at left   and right hand sides, respectively. By \eqref{eq:contradiction_with_subspace}, $\mu^n \geq \mu_{k}(I,p_n)$, so passing to the $\liminf$ and using \eqref{eq:mu_k-1_<_mu_k} we have
    \[
        \liminf_n A_n^i \geq \mu_k(I,p) - \mu_{k-1}(I,p)>0 \qquad \qquad \forall i \in \{ 0,\ldots, k-1 \}.
    \]
    The term at right hand-side has order $2$ in $a_2(\Omega_n)$, so also the term at right hand-side has the same order; hence, taking into account the fact that $\sum_{i=0}^k (\alpha_i^n)^2 = 1$, we obtain
    \[
        \alpha_i^n = O \bigl( a_2(\Omega_n)^2 \bigr) \qquad \forall i \in \{ 0,\ldots,k-1 \} \qquad \text{ and } \qquad \alpha_k^n = 1 - O \bigl( a_2(\Omega_n)^2\bigr).
    \]
    Let us now estimate the quotient in \eqref{eq:def_mu_max_con_riscalati}: we call $N$ and $D$ respectively the numerator and denominator related to the maximum.

\medskip
\noindent{\bf Numerator. }Recalling that
        \begin{align*}
            \int_{\Omega_n} \bigl((\tilde{u}_i^n)'\bigr)^2 \, dx & = \mu_i(I,p_n) \qquad  & \forall i \in \{ 0,\ldots, k \} \\
            \alpha_i^n & = O \bigl( a_2(\Omega_n)^2 \bigr) 
        \end{align*}
        and by the estimates used in \eqref{eq:optimality_condition_for_xi}, we have
        \begin{align*}
            N & =
            \sum_{i=0}^{k-1} (\alpha_i^n)^2 \mu_i(I,p_n) +
            (\alpha_k^n)^2 \Biggl \{ \int_{\tilde{\Omega}_n} \bigl( (\tilde{u}_k^n)'\bigr)^2 \, \Bigl[ 1+\tau \bigl( x_2 - c(x_1) \bigr)^2 a_2(\Omega_n)^2 \Bigr]^2 \, dx \\
            & - 
            4 \tau \Biggl( \int_{\tilde{\Omega}_n} (\tilde{u}_k^n)'(x_1) \tilde{u}_k^n(x_1) \bigl( x_2-c(x_1) \bigr) c'(x_1) \, dx \Biggr) a_2(\Omega_n)^2 \\
            & -
            4 \tau^2 \Biggl( \int_{\tilde{\Omega}_n} (\tilde{u}_k^n)'(x_1) \tilde{u}_k^n(x_1) \bigl( x_2-c(x_1) \bigr)^3 c'(x_1) \Biggr) a_2(\Omega_n)^4 \\
            & +
            4 \tau^2 \Biggl(  \int_{\tilde{\Omega}_n} \bigl( \tilde{u}_k^n (x_1) \bigr)^2 \bigl( x_2-c(x_1) \bigr)^2 c'(x_1)^2 \, dx \Biggr) a_2(\Omega_n)^4 \\
            & +
            4 \tau^2 \Biggl( \int_{\tilde{\Omega}_n} \bigl( \tilde{u}_k^n (x_1) \bigr)^2 \bigl( x_2-c(x_1) \bigr)^2 \, dx \Biggr) a_2(\Omega_n)^2 \Biggr \} \\
            & +
            2 \alpha_k^n \tau \sum_{i=0}^{k-1} \alpha_i^n \Biggl( \int_{\tilde{\Omega}_n} \bigl( \tilde{u}_i^n (x_1) \bigr)' \bigl( \tilde{u}_k^n(x_1) \bigr)'  \, \bigl( x_2 - c(x_1) \bigr)^2 \, dx \Bigg) a_2(\Omega_n)^2 \\
            & -
            4 \alpha_k^n \tau \sum_{i=0}^{k-1} \alpha_i^n \Biggl( \int_{\tilde{\Omega}_n} \bigl( \tilde{u}_i^n(x_1) \bigr)' \tilde{u}_k^n(x_1) \bigl( x_2-c(x_1) \bigr) c'(x_1) \, dx \Biggr) a_2(\Omega_n)^2 \\
            & =
            \sum_{i=0}^k (\alpha_i^n)^2 \mu_i(I,p_n) + 2 \tau \bigl( \alpha_k^n \bigr)^2 \Biggl( \int_{\tilde{\Omega}_n} \bigl( (\tilde{u}_k^n)'(x_1) \bigr)^2 \bigl( x_2 - c(x_1) \bigr)^2 \, dx \Biggr) a_2(\Omega_n)^2 \\
            & -
            4 \bigl(\alpha_k^n\bigr)^2 \tau \Biggl( \int_{\tilde{\Omega}_n} \bigl( \tilde{u}_k^n(x_1)\bigr)' \tilde{u}_k^n(x_1) \bigl( x_2 - c(x_1) \bigr) c^\prime(x_1) \, dx \Biggr) a_2(\Omega_n)^2 \\
            & +
            4 \tau^2 \bigl( \alpha_k^n \bigr)^2 \Biggl( \int_{\tilde{\Omega}_n} \bigl( \tilde{u}_k^n(x_1) \bigr)^2 \bigl( x_2 - c(x_1) \bigr)^2  \Biggr) a_2(\Omega_n)^2 + o \bigl( a_2(\Omega_n)^2 \bigr).
                  \end{align*}
        Recalling that
        \[
            \mu_i(I,p_n) \leq \mu_{k,d}^*,\qquad \sum_{i=0}^k (\alpha_i^n)^2 = 1, \qquad (\alpha_k^n)^2 = 1 - O \bigl( a_2(\Omega_n)^2 \bigr),
        \]
        we obtain
        \begin{align*}
            N
            & \leq
            \mu_{k,d}^* + \Biggl[ 2 \tau \int_{\tilde{\Omega}_n} \bigl( (\tilde{u}_k^n)' (x_1)\bigr)^2 \, \bigl( x_2 - c(x_1) \bigr)^2 \, dx - 4 \tau \int_{\tilde{\Omega}_n} \bigl( \tilde{u}_k^n (x_1) \bigr)'  \tilde{u}_k^n(x_1)  \bigl( x_2-c(x_1) \bigr) c'(x_1) \, dx \\
            & +
            4 \tau^2 \int_{\tilde{\Omega}_n} \bigl( \tilde{u}_k^n \bigr)^2 \, \bigl( x_2 - c(x_1) \bigr)^2 \, dx \Biggr] a_2(\Omega_n)^2 + o \bigl( a_2(\Omega_n)^2 \bigr).
        \end{align*}
    
\medskip
\noindent{\bf Denominator.} Recalling that $\int_{\Omega_n} (u_i^n)^2 \, dx = 1$ for each $i \in \{ 0,\ldots, k \}$, and arguing as for the numerator, we get
        \begin{align*}
            D
            & =
            \sum_{i=0}^{k-1}(\alpha_i^n)^2 + (\alpha_k^n)^2  \int_{\tilde{\Omega}_n} \bigl( \tilde{u}_k^n (x_1) \bigr)^2 \Bigl[ 1+\tau \bigl( x_2 - c(x_1) \bigr)^2 a_2(\Omega_n)^2 \Bigr]^2 \,dx  \\
            & +
            2 \sum_{i=0}^{k-1} \alpha_i^n \alpha_k^n \tau \int_{\tilde{\Omega}_n} \tilde{u}_i^n(x_1) \, u_k^n(x_1) \, \bigl( x_2-c(x_1) \bigr)^2 a_2(\Omega_n)^2 \, dx \\
            & =
            \sum_{i=0}^{k} (\alpha_i^n)^2 + 2 \tau (\alpha_k^n)^2 \Biggl( \int_{\tilde{\Omega}_n} \bigl( \tilde{u}_k^n (x_1) \bigr)^2 \bigl( x_2-c(x_1) \bigr)^2 \, dx \Biggr) a_2(\Omega_n)^2 + o \bigl( a_2(\Omega_n)^2 \bigr) \\
            & =
            1 + 2 \tau \Biggl( \int_{\tilde{\Omega}_n} \bigl( \tilde{u}_k^n (x_1) \bigr)^2 \bigl( x_2-c(x_1) \bigr)^2 \, dx \Biggr) a_2(\Omega_n)^2 + o \bigl( a_2(\Omega_n)^2 \bigr).
        \end{align*}

    Therefore we obtain
    \begin{equation}
        \label{eq:intermediate_to_cut_alpha_k^n}
        \begin{split}
            \mu_{k,d}^* & - \frac{1}{n} a_2(\Omega_n)^2 \\
            & \leq
            \frac{
                \displaystyle \mu_{k,d}^* + 2 \tau \Biggl( \int_{\tilde{\Omega}_n} \bigl( (\tilde{u}_k^n)' (x_1)\bigr)^2 \, \bigl( x_2 - c(x_1) \bigr)^2 \, dx\Biggr) a_2(\Omega_n)^2
            }{
                \displaystyle 1 + 2 \tau \Biggl( \int_{\tilde{\Omega}_n} \bigl( \tilde{u}_k^n (x_1) \bigr)^2 \bigl( x_2-c(x_1) \bigr)^2 \, dx \Biggr) a_2(\Omega_n)^2 + o \bigl( a_2(\Omega_n)^2 \bigr)
            } \\
            & -
                \frac{
                    \displaystyle 4 \tau \Biggl( \int_{\tilde{\Omega}_n} \bigl( \tilde{u}_k^n(x_1) \bigr)' \bigl( \tilde{u}_k^n(x_1) \bigr) \bigl( x_2-c(x_1) \bigr) c'(x_1) \, dx \Biggr) a_2(\Omega_n)^2
            }{
                \displaystyle 1 + 2 \tau \Biggl( \int_{\tilde{\Omega}_n} \bigl( \tilde{u}_k^n (x_1) \bigr)^2 \bigl( x_2-c(x_1) \bigr)^2 \, dx \Biggr) a_2(\Omega_n)^2 + o \bigl( a_2(\Omega_n)^2 \bigr)
            } \\
            & + \frac{
                \displaystyle 4 \tau^2 \Biggl( \int_{\tilde{\Omega}_n} \bigl( \tilde{u}_k^n (x_1) \bigr)^2 \, \bigl( x_2 - c(x_1) \bigr)^2 \, dx \Biggr) a_2(\Omega_n)^2 + o \bigl( a_2(\Omega_n)^2 \bigr)
            }{
                \displaystyle 1 + 2 \tau \Biggl( \int_{\tilde{\Omega}_n} \bigl( \tilde{u}_k^n (x_1) \bigr)^2 \bigl( x_2-c(x_1) \bigr)^2 \, dx \Biggr) a_2(\Omega_n)^2 + o \bigl( a_2(\Omega_n)^2 \bigr)
            }.
        \end{split}
    \end{equation}
    With the same computations as in the proof for $k=1$, we have
    \begin{align*}
        \biggl( \mu_{k,d}^* & - \frac{1}{n} a_2(\Omega_n)^2 \biggr) \Biggl( 1 + 2 \tau \Biggl[ \int_{\tilde{\Omega}_n} \bigl( \tilde{u}_k^n (x_1) \bigr)^2 \bigl( x_2-c(x_1) \bigr)^2 \, dx \Biggr] a_2(\Omega_n)^2 + o \bigl( a_2(\Omega_n)^2 \bigr) \Biggr) \\
        & \leq
        \mu_{k,d}^* + \Biggl[ 2 \tau\int_{\tilde{\Omega}_n} \bigl( (\tilde{u}_k^n)'(x_1)\bigr)^2 \, \bigl( x_2 - c(x_1) \bigr)^2 \, dx - 4 \tau \int_{\tilde{\Omega}_n} \bigl( \tilde{u}_k^n (x_1) \bigr)'  \tilde{u}_k^n (x_1) \, \bigl( x_2-c(x_1) \bigr) c'(x_1) \, dx \\
        & + 4 \tau^2 \int_{\tilde{\Omega}_n} \bigl( \tilde{u}_k^n (x_1) \bigr)^2 \, \bigl( x_2 - c(x_1) \bigr)^2 \, dx \Biggr] a_2(\Omega_n)^2 + o \bigl( a_2(\Omega_n)^2 \bigr) 
    \end{align*}
    that is
    \begin{equation}
        \label{eq:da:richiamare_giusta_k}
        \begin{split}
            -\frac{1}{n} a_2(\Omega_n)^2 + o\bigl( a_2(\Omega_n)^2 \bigr)
            & \leq
            2 \tau \Biggl[ \int_{\tilde{\Omega}_n} \Bigl( \bigl( (\tilde{u}_k^n)(x_1)' \bigr)^2 - \mu_{k,d}^* \bigl( \tilde{u}_k^n (x_1) \bigr)^2 \Bigr) \, \bigl( x_2-c(x_1) \bigr)^2 \, dx \\
            & -
            2 \int_{\tilde{\Omega}_n} \bigl( \tilde{u}_k^n (x_1) \bigr)' \tilde{u}_k^n(x_1) \, \bigl( x_2-c(x_1) \bigr) c'(x_1) \, dx \Biggr] a_2(\Omega_n)^2 \\
            & +
            4 \tau^2 \Biggl( \int_{\tilde{\Omega}_n} \bigl( \tilde{u}_k^n(x_1) \bigr)^2 \, \bigl( x_2-c(x_1) \bigr)^2 \, dx \Biggr) a_2(\Omega_n^2) \\
            & \eqqcolon 2 \gamma_n a_2(\Omega_n)^2 \tau + 4 \beta_n a_2(\Omega_n)^2 \tau^2,
        \end{split}
    \end{equation}
    where
    \begin{align*}
        \gamma_n & = \int_{\tilde{\Omega}_n} \Bigl( \bigl( (\tilde{u}_k^n)'(x_1) \bigr)^2 - \mu_{k,d}^* \bigl( \tilde{u}_k^n (x_1) \bigr)^2 \Bigr) \bigl( x_2-c(x_1) \bigr)^2 \, dx - 2\int_{\tilde{\Omega}_n} \bigl( \tilde{u}_k^n (x_1) \bigr)' \tilde{u}_k^n\bigl( x_2-c(x_1) \bigr) c'(x_1) \, dx, \\
        \beta_n & = \int_{\tilde{\Omega}_n} \bigl( \tilde{u}_k^n(x_1) \bigr)^2 \, \bigl( x_2-c(x_1) \bigr)^2 \, dx.
    \end{align*}
    As in the case $k=1$, we notice that the term at right hand side in \eqref{eq:da:richiamare_giusta_k} has order 2 in $a_2(\Omega_n)$ and it is a second order polynomial in $\tau$. Hence the contradiction is achieved choosing $\tau < 0$ sufficiently small provided that
    \[
        \lim_{n \to +\infty} \gamma_n \neq 0,
    \]
    that is
    \[
        \lim_{n \to +\infty}
            \int_{\tilde{\Omega}_n} \Bigl[ \bigl( \tilde{u}_k^n)'(x_1) \bigr)^2 - \mu_{k,d}^* \bigl( \tilde{u}_k^n (x_1) \bigr)^2\Bigr] \, \bigl( x_2-c(x_1) \bigr)^2 \, dx - 2 \int_{\tilde{\Omega}_n} \bigl( \tilde{u}_k^n (x_1) \bigr)' u_k^n\bigl( x_2-c(x_1) \bigr) c'(x_1) \, dx\neq 0.
    \]
    Up to a subsequence, $\tilde{u}_k^n$ converges to $\tilde{u}_k$, the $k$-th eigenfunction of the Sturm-Liouville problem related to $p_{\tilde{\Omega}}$, and it holds
    \begin{equation}
        \label{eq:limit_for_k_geq_2}
        \lim_{n \to + \infty} \gamma_n = \int_0^1 \Bigl[ \bigl( \tilde{u}_k'(x_1) \bigr)^2 - \mu_{k,d}^* \tilde{u}_k(x_1)^2  \Bigr] \biggl( \int_{S_{x_1} \cap \tilde{\Omega}} \bigl( x_2 - c(x_1) \bigr)^2 \, d\Ho^{d-1} \biggr) \, dx_1.
    \end{equation}
    
    Theorem \ref{theo:Henrot_Michetti} ensures that $p_{\tilde{\Omega}}^{\frac{1}{d-1}}$ is made by segments, so there exists a finite sequence 
    $\{z_i\}_{0\leq i\leq I}$, $I \in \N$ such that $p_{\tilde{\Omega}}^{\frac{1}{d-1}}$ is an affine function in $\bigl[ z_i,z_{i+1} \bigr]$ for $i \in \{ 0,\ldots,I -1 \}$, where $I=3$ if $d\not=3$ and $k \ge 2$, while if $d=3$ we only know
   $I \le k+1$. 
    So by \eqref{eq:momento_di_ordine_2}, for each $i \in \{ 0,\ldots,I \}$ there exists $K_i>0$ such that
  \begin{equation}\label{bgh1000}
        \int_{S_{x_1} \cap \tilde{\Omega}} \bigl( x_2 - c(x_1) \bigr)^2 \, d\Ho^{d-1} = K_i p(x_1)^{1+\frac{2}{d-1}} \qquad \forall x_1 \in \biggl[ z_i,z_{i+1} \biggr],
  \end{equation}
    and \eqref{eq:limit_for_k_geq_2} becomes
    \[
        \int_{0}^1 \Bigl[ ( \tilde{u_k}'(x_1) )^2 - \mu_{k,d}^* \tilde{u_k}(x_1)^2 \Bigr] \biggl( \int_{S_{x_1} \cap \tilde{\Omega}} \bigl( x_2-c(x_1) \bigr)^2 \, d\Ho^{d-1} \biggr) \, dx_1\hskip 2cm\]
        \[
        \hskip 2cm=
        \sum_{i=0}^{I-1}  K_i \int_{z_i}^{z_{i+1}} \Bigl[ \bigl( \tilde{u_k}'(x_1) \bigr)^2 - \mu_{k,d}^* \tilde{u_k}(x_1)^2 \Bigr] p(x_1)^{\frac{2}{d-1}}.
    \]
    
    Let us settle first the case $d\not=3$ since it is very similar to the case $k=1$ already treated. Then $I=3$, the profile of $p^\frac{1}{d-1}$ being the one of an isosceles trapezium. In particular, on the interval $[z_1,z_2]$ the function $p$ is constant and, thanks to the equation satisfied by $ \tilde u_k$ and of the fact that $ \tilde u_k (z_1)=  \tilde u_k (z_2) =0$ (see \cite[Lemma 3.11]{Henrot_Michetti}) we get 
    $$   \int_{z_1}^{z_{2}} \Bigl[ \bigl( \tilde{u_k}'(x_1) \bigr)^2 - \mu_{k,d}^* \tilde{u_k}(x_1)^2 \Bigr] p(x_1)^{\frac{2}{d-1}} =0
    $$

    With the same reasoning as in proof for $k=1$, using symmetry, we reduce the computation on the interval $[0, z_1]$ and conclude with a similar computation. 
    
    Assume now that $d=3$. We first note that all $K_i$ are equal in \eqref{bgh1000}, say $K\not=0$. Indeed, we have proved in \eqref{eq:momento_di_ordine_2},
    that $K$ seen as a function of $x_1$ is constant on each interval $[z_i,z_{i+1}]$. But since $K$ is a continuous function of $x_1$, it has to be constant on the whole interval $[0,1])$.
    Consequently,
     \[
        \int_{0}^1 \Bigl[ ( \tilde{u_k}'(x_1) )^2 - \mu_{k,d}^* \tilde{u_k}(x_1)^2 \Bigr] \biggl( \int_{S_{x_1} \cap \tilde{\Omega}} \bigl( x_2-c(x_1) \bigr)^2 \, d\Ho^{d-1} \biggr) \, dx_1\hskip 2cm\]
        \[
        \hskip 2cm=
         K \sum_{i=0}^{I-1}  \int_{z_i}^{z_{i+1}} \Bigl[ \bigl( \tilde{u_k}'(x_1) \bigr)^2 - \mu_{k,d}^* \tilde{u_k}(x_1)^2 \Bigr] p(x_1)^{\frac{2}{d-1}}= K   \int_{0}^1 \Bigl[ \bigl( \tilde{u_k}'(x_1) \bigr)^2 - \mu_{k,d}^* \tilde{u_k}(x_1)^2 \Bigr] p(x_1)^{\frac{2}{d-1}}.
    \]
Taking as tests $\tilde u_k p^\frac{1}{d-2} $ in the weak form of the equation, we get
  $$  
  \int_{0}^1 \Bigl[ \bigl( \tilde{u_k}'(x_1) \bigr)^2 - \mu_{k,d}^* \tilde{u_k}(x_1)^2 \Bigr] p(x_1)^{\frac{2}{d-1}}=     - \frac{2}{d-1} \int_0^{1} \tilde u_k'(x_1) \tilde u_k (x_1) p(x_1)^{\frac{2}{d-1}} \, p'(x_1) \, dx_1,
  $$
  $$=  -\frac{2}{d-1}\sum_{i=0}^{I-1} \int_{z_i}^{z_{i+1}} \tilde u_k'(x_1) \tilde u_k (x_1) p(x_1)^{\frac{2}{d-1}} \, p'(x_1) \, dx_1= \frac{1}{d-1}\sum_{i=0}^{I-1} \int_{z_i}^{z_{i+1}} [\tilde u_k(x_1)]^2 [p(x_1)^{\frac{2}{d-1}} \, p'(x_1)]',$$
  where in the last integration by parts we have used $\tilde u_k (z_i)=0$ for $i=1, \dots , I-1$ (see \cite[Lemma 3.11]{Henrot_Michetti}) and $p(z_0)=p(z_I)=0$. Now, on each interval $[z_i, z_{i+1}]$ the function $p^\frac{1}{d-1}$ is affine so, as for $k=1$, there exists a number $\tilde a_i$ such that
  $$ \forall x_1 \in (z_i, z_{i+1}), \;\;  p'(x_1) = \tilde{a_i} (d-1)p(x_1)^{1-\frac{1}{d-1}}. $$
  Consequently,
   \[ 
       K   \int_{0}^1 \Bigl[ \bigl( \tilde{u_k}'(x_1) \bigr)^2 - \mu_{k,d}^* \tilde{u_k}(x_1)^2 \Bigr] p(x_1)^{\frac{2}{d-1}}
        =
       K d \sum_{i=0}^{I-1}  \tilde{a}_i^2 \int_{z_i}^{z_{i+1}} \tilde{u_k}(x_1)^2 p(x_1) \, dx_1,
    \]
    that is strictly positive, which concludes the proof.
    
\end{proof}

\subsection{Proof of Theorem \ref{teo:optimality_exponent_2} }

In this Subsection we want to prove optimality of the exponent 2 in the planar case.  Let $T_{\alpha}$ be an isosceles triangle having aperture $\alpha \in (0,\pi)$ and equal sides of length $l$ such that the basis is placed along the $x_1$ axis (Figure \ref{fig_iscosceles_triangle}), i.e.
    \[
        T_{\alpha} = \biggl \{ (x_1,x_2) \in \R^2 \, : \, -l \sin \frac{\alpha}{2}  \leq x_1 \leq l \sin  \frac{\alpha}{2} , 0 \leq x_2 \leq 
        l \cos \frac{\alpha}{2} - \cot \frac{\alpha}{2}\, \lvert x_1 \rvert \biggr \}.
    \]
    If $\alpha \in \bigl( 0,\frac{\pi}{3} \bigr)$, we say $T_{\alpha}$ is subequilateral, otherwise superequilateral.

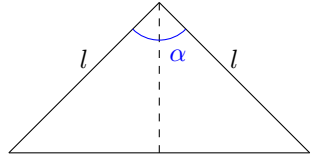
\begin{figure}[h!]
        \centering

        \begin{tikzpicture}
            \draw[black] (-2,0) -- (2,0);
            \draw[black] (2,0) -- (0,2) node[pos=0.5,above] {$l$};
            \draw[black] (0,2) -- (-2,0) node[pos=0.5,above] {$l$};

            \draw[black, dashed] (0,2) -- (0,0);

            \begin{scope}
                \clip (-2,0) -- (2,0) -- (0,2);

                \draw[blue] (0,2) circle(0.5);

                \node[blue] at (0.25,1.3) {$\alpha$};
            \end{scope}
            
        \end{tikzpicture}
        
        \caption{The triangle $T_{\alpha}$}
        \label{fig_iscosceles_triangle}
    \end{figure}

We recall the following property of $\mu_1(T_\alpha)$   proved in \cite[Proposition 6.42]{Laugesen_Siudeja}:
\begin{proposition}
    Let $T_{\alpha}$ be a superequilateral triangle, then it holds
    \begin{equation}
        \label{eq:Laugesen_Sliudeja}
        \sin^2 \biggl( \frac{\alpha}{2} \biggr) \leq \frac{\mu_1(T_{\alpha})D^2_{T_{\alpha}}}{4j_{0,1}^2} < 1.
    \end{equation}
\end{proposition}

We are in position to prove optimality of exponent $2$ in the planar case.
\begin{proof}[Proof of Theorem \ref{teo:optimality_exponent_2}]      
    Let us consider a superequilateral triangle $T_{\alpha}$ having the segment $[-\frac 12,\frac 12] \times \{ 0 \}$ as basis. Its diameter equals to $1$ and its width $w_{T_\alpha}$ in the orthogonal direction of the diameter, equals the height. Note that for any planar convex set $\Omega$
    the width $w_\Omega$ in direction orthogonal to the diameter satisfies 
   \begin{equation}
        \label{eq:equivalence_width_e_a_2_Omega}
        a_2(\Omega) \leq w_{\Omega} \leq 2a_2(\Omega).
    \end{equation}

    Formula \eqref{eq:Laugesen_Sliudeja} gives
    \[
        \mu_1(T_{\alpha}) \geq 4j_{0,1}^2 \sin^2 \biggl( \frac{\alpha}{2} \biggr).
    \]
  
%
%
%
%
    Since $w_{T_\alpha}=\ell \cos \frac{\alpha}{2}$ we can write using \eqref{eq:Laugesen_Sliudeja}
    \[
        \mu_1(T_\alpha)
        \geq
        4j_{0,1}^2 \sin^2  \biggl( \frac{\alpha}{2} \biggr)
        =
        4j_{0,1}^2 \biggl(1-\frac{w_{T_\alpha}^2}{\ell^2} \biggr).
    \]
    When $w_{T_\alpha}\to 0$,  $\alpha\to \pi$ and $\ell \to \frac{1}{2}$, using the previous inequality combined with \eqref{eq_quantitative_mu_2} for $k=1$, we can write
    \begin{equation}
        \label{bound_by_Laugesen_Sliudeja}
        4j_{0,1}^2 - 16 j_{0,1}^2 w_{T_\alpha}^2 + o(w_{T_\alpha}^2) \leq \mu_1(T_\alpha) \leq 4j_{0,1}^2 - C(1,2) w_{T_\alpha}^2,
    \end{equation}
    therefore the exponent $2$ is optimal.
\end{proof}

\section{An explicit constant}
\label{section_4}
As explained in the Introduction, the proof of our quantitative inequality being done by contradiction does not provide any value
for the constant $C(k,d)$. This is the same situation than in most of the known quantitative inequalities.
Therefore, in this section, we want to show on an example, that we can find an explicit value for the constant $C(1,2)$.
For simplicity, we choose to work with a planar convex domain $\Omega$ that has two orthogonal axis of symmetry: the mediatrix of its diameter being one of them.
Without loss of generality, by scaling and rotating, we can assume that the diameter satisfies $D_\Omega=1$ and it is located on the interval $(0,1)\times \{0\}$.
Therefore, $\Omega$ will be represented as
$$\Omega:=\{(x,y)\in \mathbb{R}^2, 0<x<1, -h^+(x)< y < h^+(x) \}$$
where $h^+(x)$ is a concave function and $\tilde{h}(x)=2h^+(x)$ is the profile function of $\Omega$.
Here is our result:
\begin{theorem}\label{theoexplicit}
Let $\Omega$ be a planar convex domain, with two orthogonal axis of symmetry. the mediatrix of its diameter being one of them. Then it holds
\begin{equation}\label{mu1explicit}
\mu_1(\Omega) \leq \frac{4 j_{0,1}^2}{D_\Omega^2} - 0.108 \frac{w_\Omega^2}{D_\Omega^4}.
\end{equation}
In this inequality, $D_\Omega$ is the diameter of $\Omega$ while $w_\Omega$ is its width in the direction orthogonal to a diameter.
\end{theorem}
\begin{proof}
The profile function of the convex set is denoted by $\tilde{h}(x)$, it is concave, symmetric with respect to $x=1/2$ and increasing on $(0,1/2)$.
Moreover $\tilde{h}(1/2)=w_\Omega:=w$. We will normalize it by setting $h(x)=\tilde{h}(x)/w$ and $h$ takes values between $0$ and $1$.
Let us introduce the function
$$f(x)=\left\lbrace
\begin{array}{cc}
J_0(2j_{0,1} x) & \mbox{if } x\leq \frac{1}{2}\\ \vspace{2mm}
-J_0(2j_{0,1} (1-x)) & \mbox{if } x\geq \frac{1}{2}\\
\end{array}\right.$$
This function is odd with respect to $x=1/2$ and is actually the first (non-constant) eigenfunction for the Sturm-Liouville problem with the function $p(x)=x$ on $(0,\frac{1}{2})$ and $p(x)=1-x$ on $(\frac{1}{2},1)$.
Now we are considering as a test function for the Neumann problem
$$v(x,y)=f(x)(1+\tau y^2)$$
where $\tau$ is a (negative) constant that we define now.

We need to introduce several integrals involving Bessel functions, namely for any integer $p$ (we will only use $p=0,1,3,5$ in the sequel)
and any function $h(x)$:
$$I_{p,0}(h) = \int_0^{\frac{1}{2}} h(x)^p(J_0(2j_{0,1}x))^2 dx, \qquad I_{p,1}(h) = \int_0^{\frac{1}{2}} h(x)^p(J_1(2j_{0,1}x))^2 dx.$$
Let us also consider
$$\psi(p,h)=I_{p,0}(h) - I_{p,1}(h)=\int_0^{\frac{1}{2}} h(x)^p(J_0(2j_{0,1}x)^2 - J_1(2j_{0,1}x)^2) dx.$$
Now we define $\tau$ (for reasons that will become clear at the end) as
\begin{equation}\label{deftau}
\tau = j_{0,1}^2 \frac{\psi(3,2x)}{I_{0,0}(1)} \simeq -0.569.
\end{equation}
Let us come back to our test function $v$.
Since $f(x)$ is odd with respect to $1/2$ and $\Omega$ has the line $x=1/2$ as axis of symmetry, we immediately have
$\int_\Omega v(x,y) dxdy=0$ that makes $v$ admissible for the variational formulation.  Therefore
$$\mu_1(\Omega) \leq \dfrac{\int_\Omega |\nabla v|^2 dX}{\int_\Omega v^2 dX}=\dfrac{\int_{\Omega_\ell} |\nabla v|^2 dX}{\int_{\Omega_\ell} v^2 dX}$$
where $\Omega_\ell=\Omega \cap \{x_1\leq 1/2\}$ is the half left part of $\Omega$ and the equality above stands by symmetry.
In computing the above integrals, we will use
\begin{eqnarray*}
\int_{-h^+(x)}^{h^+(x)} y^2dy = \frac{2}{3} {h^+(x)}^3=\frac{1}{12} (\tilde{h}(x))^3 \\
\int_{-h^+(x)}^{h^+(x)} y^4dy =\frac{2}{5} {h^+(x)}^5=\frac{1}{80} (\tilde{h}(x))^5 \\
\end{eqnarray*}
Therefore, we get (we use here the fact that $J_0^\prime=-J_1$):
$$
\mu_1(\Omega) \leq \dfrac{4j_{0,1}^2 \int_0^{\frac{1}{2}} (J_1(2j_{0,1}x)^2\left\lbrack \tilde{h}(x)+\frac{2\tau}{12}\tilde{h}(x)^3+\frac{\tau^2}{80}\tilde{h}(x)^5\right\rbrack dx +\frac{4\tau^2}{12} \int_0^{\frac{1}{2}} (J_0(2j_{0,1}x)^2 \tilde{h}(x)^3 dx}{\int_0^{\frac{1}{2}} (J_0(2j_{0,1}x)^2\left\lbrack
\tilde{h}(x)+\frac{2\tau}{12}\tilde{h}(x)^3+\frac{\tau^2}{80}\tilde{h}(x)^5\right\rbrack}.
$$
Using the normalization with $h(x)$ and the notations for the different integrals, we can rewrite the previous inequality as
\begin{eqnarray*}
\mu_1(\Omega) \leq \dfrac{4j_{0,1}^2 \left\lbrack I_{1,1}(h)w + \frac{2\tau}{12} I_{3,1}(h) w^3+\frac{\tau^2}{80} I_{5,1}(h) w^5\right\rbrack +
\frac{4\tau^2}{12} I_{3,0}(h) w^3}{I_{1,0}(h) w + \frac{2\tau}{12} I_{3,0}(h) w^3+\frac{\tau^2}{80} I_{5,0}(h) w^5}\\
=4j_{0,1}^2 \dfrac{ I_{1,1}(h) + \frac{2\tau}{12} I_{3,1}(h) w^2+\frac{\tau^2}{80} I_{5,1}(h) w^4 + \frac{\tau^2}{12 j_{0,1}^2} I_{3,0}(h) w^2}
{I_{1,0}(h)  + \frac{2\tau}{12} I_{3,0}(h) w^2+\frac{\tau^2}{80} I_{5,0}(h) w^4}.
\end{eqnarray*}
We want to find an explicit constant $C$ such that $\mu_1(\Omega) \leq 4j_{0,1}^2(1-C w^2)$ that is equivalent to prove the following inequality
\begin{equation}\label{eqex1}
C\leq \frac{1}{w^2} \frac{\psi(1,h) +\frac{2\tau}{12} \psi(3,h)w^2 +\frac{\tau^2}{80} \psi(5,h) w^4-\frac{\tau^2}{12 j_{0,1}^2} I_{3,0}(h) 
w^2}{I_{1,0}(h)  + \frac{2\tau}{12} I_{3,0}(h) w^2+\frac{\tau^2}{80} I_{5,0}(h) w^4}
\end{equation}
for all $w\in (0,1]$ and all $h$ concave (taking values between $0$ and $1$).
Let us first look at the denominator $D$ that can be written
$$D(h)=\int_0^{\frac{1}{2}}  [h(x)+\frac{\tau}{6} w^2 h^3(x) +\frac{\tau^2}{80} w^4 h^5(x)]J_0(2j_{0,1}x)^2 dx$$
and since $z\mapsto z+\frac{\tau}{6} w^2 z^3 +\frac{\tau^2}{80} w^4 z^5$ is increasing (its derivative is $(1+\frac{\tau}{4} z^2w^2)^2$) and $h(x)\leq 1$
 we infer that $D(h)\leq D(1)$ with
 \begin{equation}\label{denoexpli}
 D(1)= \left(1+\frac{\tau}{6} w^2 +\frac{\tau^2}{80} w^4\right) I_{0,0}.
 \end{equation}
In the same way the last term in the numerator can be estimated by
\begin{equation}\label{lastermexpli}
\frac{\tau^2}{12 j_{0,1}^2} I_{3,0}(h) \leq \frac{\tau^2}{12 j_{0,1}^2} I_{3,0}(1)=\frac{\tau^2}{12 j_{0,1}^2} I_{0,0}.
\end{equation}
Now we concentrate on the three first terms in the numerator. We want to minimize their sum, with respect to the function $h$ that leads
us to solve the following problem in the calculus of variations. We denote by $J$ the sum:
$$J(h)=\psi(1,h) +\frac{\tau}{6} \psi(3,h)w^2 +\frac{\tau^2}{80} \psi(5,h) w^4$$
here,  $w$ is a fixed number between $0$ and $1$ and $\tau$ has been defined in \eqref{deftau}.
\begin{proposition}
Let us consider the following minimization problem
\begin{equation}\label{calvaexpli}
\min \{J(h), h:[0,\frac{1}{2}] \to [0,1],  \mbox{ $h$ concave }, h(0)\geq 0, h(1/2)=1, h^\prime(1/2)\geq 0\}.
\end{equation}
Then the solution of \eqref{calvaexpli} is the function $h(x)=2x$.
\end{proposition}
\begin{proof}
We can rewrite $J(h)$ as
$$J(h)=\int_0^{\frac{1}{2}}  [h(x)+\frac{\tau}{6} w^2 h^3(x) +\frac{\tau^2}{80} w^4 h^5(x)](J_0(2j_{0,1}x)^2 - J_1(2j_{0,1}x)^2)dx .$$
Let us introduce the function $g(x)=J_0(2j_{0,1}x)^2-J_1(2j_{0,1}x)^2$, see its graph on Figure \ref{graphg}.
\begin{figure}
\begin{center}
\includegraphics[scale=0.25]{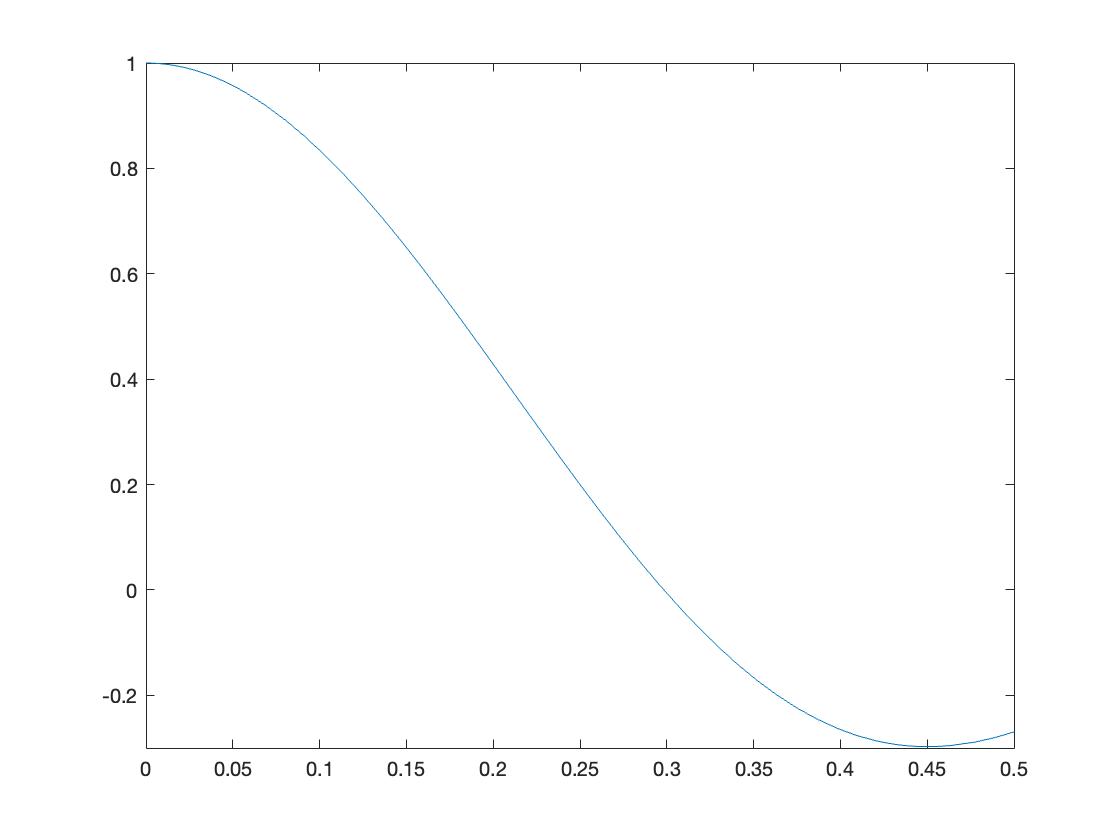} \caption{The graph of $x\mapsto J_0(2j_{0,1}x)^2-J_1(2j_{0,1}x)^2$}\label{graphg}
\end{center}
\end{figure}
The proof we give below works for any $\tau \in (-1,1)$.  What makes the result non trivial is the fact that the function $g$ changes its sign
at one point $x_0$.

The existence of a minimizer is standard, a minimizing sequence of concave function converges uniformly on every compact subset of 
the interval $I=(0,1/2]$
(because its derivative will be uniformly bounded on such a compact) and the limit $h^*$ is concave and satisfies also $h^*(0)\geq 0, h^*(\frac{1}{2})=1$.
Note that $h^\prime(1/2)\geq 0$ is equivalent (by concavity of $h$) to $h$ increasing, and this also pass to the limit;
Therefore $h^*$ is a minimizer.

 Let us write the optimality condition, using the classical framework explained for example in 
\cite{Lam-Nov2010} or \cite{Henrot_Michetti}. First, the derivative of $J$ with respect to $h$ is given by
$$<dJ(h^*),v>= \int_0^{\frac{1}{2}}  (1+\frac{\tau}{4} {h^*(x)}^2 w^2)^2 v(x) g(x) dx.$$
Therefore, there exists a Lagrange multiplier $\xi \in H^1(I)$ (corresponding to the concavity constraint) with $\xi \geq 0$
and $\xi=0$ on the support $S$ of $h^{\prime\prime}$,  and there exists three numbers $\alpha_0$, $\alpha_1$
and $\alpha_2$ (corresponding to the constraints $h(0)\geq 0, h(1)=1$) and $h^\prime(1/2)\geq 0$ such that
\begin{equation}\label{optiexpli}
\xi^{\prime\prime}= (1+\frac{\tau}{4} {h^*(x)}^2 w^2)^2  g(x) + \alpha_0\delta_0 + \alpha_1 \delta_{1/2} - \alpha_2 \delta^\prime(1/2).
\end{equation}
We denote by $f(x)=(1+\frac{\tau}{4} {h^*(x)}^2 w^2)^2  g(x)$ and therefore, we have $\xi^{\prime\prime}=f(x)$ on $I$.
The previous relation is valid on the whole interval $I$ and shows, in particular, that $\xi$ is a $C^2$ function in the open interval.
The support $S$ being closed, its complement $S^c$ is a union of open intervals. Let us denote by $(a,b)$ with $0<a<b<1/2$ a possible
such interval in $S^c$ with the assumption that $a$ and $b$ belongs to $S$. In particular $\xi(a)=\xi(b)=0$.
Moreover, since $\xi$ is $C^2$ and must remain non-negative outside $[a,b]$, this implies that $\xi^\prime(a)=\xi^\prime(b)=0$.
But it is easy to check that the four relations $\xi(a)=\xi(b)=\xi^\prime(a)=\xi^\prime(b)=0$ would imply that the function $f$
vanishes (at least) twice on the interval $(a,b)$. Now since $|\tau|<1, 0\leq h^*\leq 1$ and $ 0\leq w\leq 1$, we see that $(1+\tau {h^*(x)}^2 w^2)$
does not vanish and since $g$ vanishes only once, this situation is not possible: 
$S^c$ does not contain any interior interval. On the other hand $S$ does not contain neither any interval (where $\xi$ would be identically 0).
Therefore, the only possibilities are:
\begin{description}
\item[(a)] $S=\{0,a,1/2\}$ for some point $a$ in the interior of the interval, and $h^*$ is affine on each interval $(0,a)$ and $(a,1/2)$.
\item[(b)] $S=\{0,1/2\}$ and $h^*$ is affine on $I$.
\end{description}
We claim that we can rule out the situation where $h^*(0)=h_0>0$ in both cases. Indeed, in that case, considering a deformation $v(x)$
defined by (we recall that $x_0$ is the zero of the function $g$)
$$v(x)=\left\lbrace
\begin{array}{cc}
x-x_0 & \mbox{ if } x\leq x_0\\
0 & \mbox{ if } x\geq x_0\\
\end{array}\right.$$
we see that $h+tv$ is still concave for $t$ small and the derivative $<dJ(h^*),v>$ is clearly negative, yielding a non optimal $h^*$.

It remains to rule out the case of a function $h^*$ piecewise affine with a corner at some point $a\in (0,1/2)$.
First of all, since $\xi$ must remain non-negative and $\xi(a)=\xi^\prime(a)=0$ for the same reasons than before, we must have
$\xi^{\prime\prime}(a)\geq 0$. In view of the expression given in \eqref{optiexpli}, this implies that $a\leq x_0$. The case $a=x_0$
is not possible either since then we would have $\xi^{\prime\prime}(a)=0$ and a negative third derivative at $x_0$ contradicting
the sign condition for $\xi$. Therefore $a<x_0$.
Now, let us consider the following perturbation
$$v(x)=\left\lbrace
\begin{array}{cc}
-x(x_0 -a) & \mbox{ if } x\leq a\\
a(x-x_0) & \mbox{ if } a\leq x\leq x_0\\
0 & \mbox{ if } x\geq x_0\\
\end{array}\right.$$
this function $v$ is continuous and always negative, the perturbation $h^*+tv$ remains concave for $t>0$ small (because the negative Dirac mass
of $h^*$ at point $a$ absorbs the positive Dirac mass of $v$), and $h^*+tv\leq h^*\leq 1$ therefore one more time, we have 
$<dJ(h^*),v>$  negative, yielding a non optimal $h^*$. As a conclusion, the only remaining possibility is $h^*(x)=2x$.
\end{proof}

As a consequence of the previous proposition, we have the following inequality
\begin{equation}\label{numexpli2}
\psi(1,h) +\frac{\tau}{6} \psi(3,h)w^2 +\frac{\tau^2}{80} \psi(5,h) w^4 \geq 
\psi(1,2x) +\frac{\tau}{6} \psi(3,2x)w^2 +\frac{\tau^2}{80} \psi(5,2x) w^4 .
\end{equation}
Now, it turns out that a classical relation of Bessel function is
$$\int_0^{j_{0,1}}  t (J_0(t)^2 - J_1(t)^2)dt =0 $$
that implies that $\psi(1,2x)=0$. Therefore, taking into account \eqref{numexpli2},  \eqref{lastermexpli}, \eqref{denoexpli}, the quotient 
in \eqref{eqex1} can be estimated from below by:
\begin{equation}\label{eqex2}
Q(w):=\frac{\frac{\tau}{6} \psi(3,2x) +\frac{\tau^2}{80} \psi(5,2x) w^2 - \frac{\tau^2}{12 j_{0,1}^2} I_{0,0} }{ \left(1+\frac{\tau}{6} w^2 +\frac{\tau^2}{80} w^4\right) I_{0,0}}
\end{equation}
Now, an elementary study of the (positive) function $w\mapsto Q(w)$ shows that it is increasing on the interval $[0, 1]$, thus its minimal value is
obtained for $w=0$ and 
$$Q(0)= \left(\frac{\tau}{6} \psi(3,2x) - \frac{\tau^2}{12 j_{0,1}^2} I_{0,0}\right)/I_{0,0}.$$
The choice of $\tau$ that we did at the beginning yields the better (maximal) value for $Q(0)$
that is 
$$M:=\frac{1}{12} j_{0,1}^2 \frac{\psi(3,2x)^2}{I_{0,0}^2} \simeq 0.004669.$$
Multiplying this value by $4j_{0,1}^2$ gives the value of the constant $0.108$ stated in the theorem.
\end{proof}

 \section*{Acknowledgments} D.B. and A.H. would like to thank the Isaac Newton Institute for Mathematical Sciences, Cambridge, for support and hospitality during the programme {\it Geometric Spectral Theory} where work on this paper was undertaken.   This work was supported by EPSRC grant no EP/Z000580/1. They were 
 also partially supported by the ANR project STOIQUES financed by the French Agence Nationale de la Recherche (ANR-24-CE40-2216).  D.B. was partially supported by a grant from the Simons Foundation. 
  
\medskip
\noindent Data Availability Data sharing is not applicable to this article as no datasets were
generated or analysed.\\
The authors declare that there are no conflict of interest with this work.

\bibliographystyle{abbrv}
\bibliography{biblio}

\end{document}